\documentclass[a4paper,12pt]{article}

\usepackage{natbib,enumerate}
\usepackage{amsmath}
\usepackage{amsthm}
\usepackage{amssymb}
\usepackage{hyperref}
\usepackage{amsfonts}
\usepackage{amscd}
\usepackage[utf8]{inputenc}

\usepackage{verbatim}
\usepackage{tikz}

\newtheorem{theorem}{Theorem}

\newtheorem{definition}[theorem]{Definition}

\newtheorem{lemma}[theorem]{Lemma}

\newtheorem{proposition}[theorem]{Proposition}
\theoremstyle{definition}

\theoremstyle{remark}
\newtheorem{example}[theorem]{Example}
\newtheorem{remark}[theorem]{Remark}

\newcommand* \argmin {\mathop{\rm argmin}\nolimits }

\newcommand* \ucore {\mathop{\ \rm \mathbf{u}\text{-core}}\nolimits}
\newcommand* \core {\mathop{\rm core}\nolimits}

\title{TU-games with utilities: the prenucleolus and its characterization set\thanks{The authors thank for the comments of the anonymous referees, the participants of SING 2021, VOCAL 2022, EURO 2022, EUROpt 2023, among others. This research was supported by the Hungarian Scientific Research Fund (K 146649), and by the Pro Progressio Foundation.}}

\author{Zsófia Dornai\thanks{Department of Analysis and Operations Research, Institute of Mathematics, Budapest University of Technology and Economics, Műegyetem rkp. 3., H-1111 Budapest, Hungary,  dornaizs@math.bme.hu} and Miklós Pintér\thanks{Corresponding author: Corvinus Center for Operations Research, Corvinus Institute of Advanced Studies, Corvinus University of Budapest, pmiklos@protonmail.com.}}

\begin{document}

\maketitle

\begin{abstract}
TU-games with utility functions are considered. Generalizations of the prenucleolus, essential coalitions and the core: the $\mathbf{u}$-prenucleo\-lus, $\mathbf{u}$-essential coalitions and the $\mathbf{u}$-core respectively are introduced. We show that $\mathbf{u}$-essential coalitions form a characterisation set for the $\mathbf{u}$-prenucleolus in case of  games with nonempty $\mathbf{u}$-core.

\bigskip

\end{abstract}


\section{Introduction}

The (pre)nucleolus \citep{Schmeidler1969} is a widely used solution concept of transferable utility cooperative games. Its practical importance lies in the fact that its aim is to minimize the dissatisfaction of the most dissatisfied coailition, which makes it unequivocally interesting for social, financial and technical usages as well. A considerable evidence of its value is that it has numerous variants and generalizations, such as the percapita prenucleolus \citep{Grotte1970,Grotte1972}, the weighted prenucleolus \citep{DerksHaller1999}, the modified nucleolus \citep{Sudholter1996,Sudholter1997} and the general nucleolus \citep{PottersTijs1992,MaschlerPottersTijs1992}. From the viewpoint of this research the most abstract generalization is the one by \citet{PottersTijs1992} and \citet{MaschlerPottersTijs1992} called the general nucleolus. The last two papers focus on the characterization of the general nucleolus, while also providing a generalization of the lexicographic center algorithm \citep{MaschlerPelegShapley1979}, and Kohlberg's theorem \citep{Kohlberg1971}. 

In this paper, we consider the restricted cooperation case, and in connection with these games we introduce a special case of the general nucleolus, called the $\mathbf{u}$-prenucleolus. Using this less general setting, we can also generalize the results by \citeauthor{Katsev2013} to the $\mathbf{u}$-prenucleolus; namely, we give sufficient and necessary conditions for the $\mathbf{u}$-prenucleolus to be nonempty and to be single-valued.


Moreover, the $\mathbf{u}$-prenucleolus is a common generalization of some well known variants of the prenucleolus such as the percapita prenucleolus \citep{Grotte1970,Grotte1972}, the $q$-prenucleolus \citep{Solymosi2019} and the weighted prenucleolus \citep{DerksHaller1999} among others. 

All of these generalizations are common in using a modified version of the excesses. In our case, we define a $\mathbf{u}$ function on the excesses. Using this $\mathbf{u}$ function, we also define a generalization of balanced games and the core \citep{Shapley1955,Gillies1959}: the $\mathbf{u}$-balanced games and the $\mathbf{u}$-core respectively. Furthermore, we show that a game is $\mathbf{u}$-balanced if and only if its $\mathbf{u}$-core is nonempty, that is, among others, we generalize the Bondareva--Shapley theorem \citep{Bondareva1963,Shapley1967,Faigle1989}. 

Another variant of the prenucleolus, which is worth mentioning, because of its similarity to the $\mathbf{u}$-prenucleolus, is the modified nucleolus \citep{Sudholter1996,Sudholter1997}. It also applies the idea of defining a modified excess vector and finding the payoff vectors, which lexicographically minimize it. However, this modification cannot be described by using an above-mentioned utility function, so it is not a special case of our approach, the $\mathbf{u}$-prenucleolus.

Computing the prenucleolus or its variants and generalizations is time-consuming. Although it has been known that $2n -2$ coalitions are enough to characterize the prenucleolus of an $n$ player game  \citep{Brune1983,ReijniersePotters1998} -- such a set of coalitions is called a characterization set --, it is not easy to find these $2n-2$ coalitions.

There are several algorithms to compute the (pre)nucleolus \citep{Kohlberg1972, Solymosi1993, PereaPuerto2013}, however these algorithms have a complexity of $\mathcal{O} (2^n)$ in general. On the other hand, for some classes of games like neighbour games \citep{Hamers2003}, permutation games under certain conditions \citep{Solymosi2005}, tree games \citep{Maschler2010}, a large class of directed acyclic graph games \citep{Sziklai2017} the nucleolus can be computed in polynomial time in the number of players. In addition, some heuristics provide efficient algorithms for us to find an allocation close to the nucleolus \citep{PereaPuerto2019}. 

In this paper, we consider the lexicographic center algorithm \citep{MaschlerPelegShapley1979}, which calculates the nucleolus using $\mathcal{O} (n)$ LPs with $\mathcal{O} (n)$ variables and $\mathcal{O} (2^n)$ constraints. The number of constraints can be reduced by finding a smaller characterization set for the nucleolus than the one including all the feasible coalitions. 


\citet{Huberman1980} showed that the so-called essential coalitions give a characterization set for the nucleolus of balanced games. In certain classes of games (e.g. matching-games) the cardinality of essential coalitions is polynomial in the number of players and they are also easy to find, thereby providing a way of computing the nucleolus in polynomial time. On the other hand, Huberman's theorem, typically, cannot be applied to the variants and generalizations of the prenucleolus, neither can it be used for non-balanced games.

In this paper, we define the $\mathbf{u}$-essential coalitions as a generalization of essential coalitions and prove that the $\mathbf{u}$-essential coalitions give a characterization set for the $\mathbf{u}$-prenucleolus of $\mathbf{u}$-balanced games. Choosing our $\mathbf{u}$ function accordingly, this generalization of Huberman's theorem can be applied to give a characterization set for the prenuclolus of non-balanced games or to give a characterization set for the percapita prenucleolus of balanced games, among others; thereby solving the above-mentioned problems.

Before summarizing our main results stated in the paper, let us consider some alternative approaches for defining the $\mathbf{u}$-excess. We define the $\mathbf{u}$-excess by applying a utility function over the excess, however, it is also possible to take the excess of the utilities instead. This approach would also generalize the prenucleolus and the percapita prenucleolus. In addition, using the identity or the percapita utility functions, the two above-defined $\mathbf{u}$-excesses would coincide, due to the linearity of the utility functions. However, in case of a non-linear utility function, the two approaches would differ. Taking the excess of the utililities instead of the utility of the excess is an interesting approach, however, in this paper, we stick to using the utility of the excess, since we are interested in the idea, that different coalitions might value their excesses differently. 

The literature also pushes us to our choice. The general nucleolus \citep{PottersTijs1992,MaschlerPottersTijs1992} also applies functions on the excesses. The results stated in this paper only hold in case of taking the utility of the excess and not in case of taking the excess of the utilities. This other idea can be a fuel for future work.

One of the main contributions of this paper is the notion of the $\mathbf{u}$-prenucleolus, which is a generalization of the prenucleolus and the percapita prenucleolus. It is also a special case of the general prenucleolus. In this setup, we define the generalization of the core, least-core, balanced games: the $\mathbf{u}$-core, the $\mathbf{u}$-least-core and $\mathbf{u}$-balanced games, respectively. A respectable achievement of our researches is that using these notions, we prove the generalizations of the Bondareva--Shapley theorem and two theorems from Katsev and Yanovskaya to the $ \mathbf{u}$-prenucleolus, which, according to our present knowledge, cannot be proven to the general prenucleolus. 
We also introduce the notion of $\mathbf{u}$-essential coalitions, and show that the class of $\mathbf{u}$-essential coalitions form  a characterization set for the $\mathbf{u}$-prenucleolus of $\mathbf{u}$-balanced games. In other words, we generalize \citet{Huberman1980}'s result to TU-games with utilities. 

The setup of the paper is as follows: In Section \ref{sec:Pre}, we discuss the basic concepts and notations used in the paper; in Section \ref{sec:u}, we introduce the notions of TU-games with utilities, $\mathbf{u}$-excess, $\mathbf{u}$-prenucleolus and $\mathbf{u}$-core; in Section \ref{sec:balanced}, we define the so-called $\mathbf{u}$-balanced games, and show that a game is $\mathbf{u}$-balanced if and only if its $\mathbf{u}$-core is nonempty. 

In Section \ref{sec:lexi}, we give a generalization of the lexicographic center algorithm for calculating the $\mathbf{u}$-prenucleolus -- this algorithm is a special case of the lexicographic center algorithm for calculating the general nucleolus by \cite{MaschlerPottersTijs1992}. 

In Section \ref{sec:existence}, we generalize \citet{Katsev2013}'s theorem, namely, we give a sufficient and necessary condition for the nonemptyness of the $\mathbf{u}$-prenucleolus. In Section \ref{sec:unicity}, we generalize another theorem by \citeauthor{Katsev2013}, namely, we give a sufficient and necessary condition for the single-valuedness of the $\mathbf{u}$-prenucleolus. 

In Section \ref{sec:essential}, we define the so-called $\mathbf{u}$-essential coalitions, and generalize \citet{Huberman1980}'s theorem by proving that the $\mathbf{u}$-essential coalitions form a characterization set for the $\mathbf{u}$-prenucleolus of $\mathbf{u}$-balanced games. In Section \ref{sec:invariance}, we describe the $\mathbf{u}$ functions for which the prenucleolus and the core coincide with the $\mathbf{u}$-prenucleolus and the $\mathbf{u}$-core, respectively. 

In Section \ref{sec:application}, we show that in case of assignment games, using the reciprocal percapita utility function, we get polynomial many $\mathbf{u}$-essential coalitions in the number of players. Finally, the last section offers a brief conclusion.

\section{Preliminaries}\label{sec:Pre}


Given a nonempty finite set of players $N$ and a characteristic function $v \colon 2^{N} \to \mathbb{R}$ such that $v(\emptyset )=0$, $v$ is called a TU-\emph{game}. Let $\mathcal{G}^N$ denote the class of TU-games with player set $N$; additionally, let the set of coalitions be denoted by $\mathcal{P}(N) := \{ S\subseteq N \} $, and the set of non-trivial coalitions be denoted by $\mathcal{P}^*(N):=\{ S\subseteq N\colon S\neq \emptyset , S\neq N\}$. Furthermore, let $\mathcal{D}_S$ denote the class of partitions of set $S\subseteq N$ but $\{ S\}$.


Let $\mathcal{A} \subseteq \mathcal{P} (N)$ be such that $\emptyset,N \in \mathcal{A}$, then $\mathcal{A}$ is called a set of feasible coalitions.  In this case, the characteristic function $v \colon \mathcal{A} \to \mathbb{R}$ is called a TU-game with restricted cooperation. Let $\mathcal{G}^{N,\mathcal{A}}$ denote the class of TU-games (henceforth: game) with feasible coalitions $\mathcal{A}$. If $\mathcal{A} = \mathcal{P} (N)$, then $\mathcal{G}^{N,\mathcal{A}} = \mathcal{G}^N$, so every introduced concept for games with restricted cooperation is a generalization of a concept for classical TU-games.

Let $\mathcal{A}^*=\mathcal{A}\setminus \{ N, \emptyset \}$ denote the set of feasible, non-trivial coalitions and let $\mathcal{D}_S^{A^*} = \{ B\in \mathcal{D}_S \colon B\subseteq \mathcal{A}^* \}$ denote the $\mathcal{A}^*$-partition set of $S\in \mathcal{A}^*$.


A solution is a set-valued mapping from a set of games with player set $N$ to $\mathbb{R}^N$. Widely used solutions in the literature are the core \citep{Shapley1955,Gillies1959}, the kernel \citep{DavisMaschler1965}, and the  bargaining set \citep{AumannMaschler1964} among others. A value is a singleton valued solution; well-known values in the literature are the Shapley-value \citep{Shapley1953} and the (pre)nucleolus \citep{Schmeidler1969} among others.


Let $I(v) := \{ x\in \mathbb{R}^{N} \colon \sum_{i \in N}x_i = v(N) \text{ and } x_i\geq v(\{ i\} ) \text{ } \forall \{i\} \in \mathcal{A}\}$ and $I^*(v) := \{ x\in \mathbb{R}^{N} \colon \sum_{i \in N}x_i = v(N)\}$ denote the set of imputations and preimputations of a game $v \in \mathcal{G}^{N, \mathcal{A}}$, respectively. 

Given a game $v \in \mathcal{G}^{N,\mathcal{A}}$, a coalition $S \in \mathcal{A}$ and a payoff vector $x \in \mathbb{R}^N$, the \emph{excess} of coalition $S$ by the payoff vector $x$ in the game $v$ is $e (S, x) := v(S) -x(S)$, where $x(S) := \sum_{i\in S}x_i$.

The core of a game $v \in \mathcal{G}^{N,\mathcal{A}}$ is the set of preimputations for which the excess of any feasible coalition is non-positive:
\begin{equation*}
\text{core}(v) := \left\{ x\in \mathbb{R}^{N} \colon x (N) = v(N) \text{ and } e(S, x) \leq 0, \forall S \in \mathcal{A}^\ast \right\}.
\end{equation*}
In case the core of a game is nonempty, we say that the game is balanced.

Given a game $v \in \mathcal{G}^{N,\mathcal{A}}$ and a payoff vector $x \in \mathbb{R}^N$, the excess vector by the payoff $x$ in the game $v$ is the vector containing all the excesses in non-increasing order, that is $E(x) := (e (S,x))_{S \in \mathcal{A}^\ast} \in \mathbb{R}^{|\mathcal{A}^\ast|}$, where $E (x)_i \geq E(x)_j$ if $i \leq j$.

The lexicographical ordering between $x, y \in \mathbb{R}^n$ is the following: we say that $x \leq _L y$ if $x=y$ or if there exists a $k$, such that $x_k < y_k$ and $x_i = y_i$ for every $i<k$. 

The nucleolus is the set of imputations which lexicographically minimize the excess vector over the set of imputations, that is, $N(v) = \{ x \in I(v) \colon E(x) \leq _L E(y) \ \forall y\in I(v) \}$; and the prenucleolus is the set of preimputations which lexicographically minimize the excess vector over the set of preimputations, that is, $N^*(v) = \{ x \in I^*(v) \colon E(x) \leq _L E(y) \ \forall y\in I^*(v) \}$.



A set of coalitions $\mathcal{S} \subseteq \mathcal{A}$ is a balanced set system if there exists a balancing weight system $\lambda_S \in \mathbb{R}_+$, $S \in \mathcal{S}$ such that

\begin{equation*}
\sum \limits_{S \in \mathcal{S}} \lambda_S \, \chi_S = \chi_N,
\end{equation*}

\noindent where $\chi_T \in \mathbb{R}^N$ is the characteristic vector of set $T$.



\section{Games with utility functions}\label{sec:u}

A well-known variant of the prenucleolus is the percapita prenucleolus \citep{Grotte1970, Grotte1972}. The percapita prenucleolus differs from the preucleolus in a way that instead of using the excesses, it uses the so-called percapita excesses. The percapita excess of a coalition $S \in \mathcal{A}^\ast$ of a game $v \in \mathcal{G}^{N,\mathcal{A}}$ with a payoff vector $x \in \mathbb{R}^N$ is $\frac{e(S, x)}{|S|}$. Similarly, the percapita excess vector is $E_{pc} (x) := (\frac{e (S,x)}{|S|})_{S \in \mathcal{A}^\ast} \in \mathbb{R}^{|\mathcal{A}^\ast|}$, where $E_{pc} (x)_i \geq E_{pc} (x)_j$ if $i \leq j$. Accordingly, the percapita prenucleolus is defined as follows: $N^\ast_{pc} = \{ x\in I^*(v) \colon E_{pc} (x) \leq _L E_{pc} (y) \ \forall y\in I^*(v)\}$.

\cite{Solymosi2019} considers a further generalization of the percapita prenucleolus, where, instead of dividing the excess by the cardinality of $S$, it is divided by $q(S)$, where $q$ is a real valued function over the feasible coalitions. Thereby, \cite{Solymosi2019} introduced the notion of the $q$-nucleolus $N_q$, which is defined as follows: $N^\ast_{q} = \{ x\in I^*(v) \colon E_{q} (x) \leq _L E_{q} (y) \ \forall y\in I^*(v)\}$, where the $q$-excess vector is defined as $E_{q} (x) := (\frac{e (S,x)}{q(S)})_{S \in \mathcal{A}^\ast}$, where $E_q (x)_i \geq E_q (x)_j$  if $i \leq j$.

Partially inspired by the above generalizations of the prenucleolus we generalize the prenucleolus further by introducing functions, called utility functions, applied to the excesses. Formally, see the following definition. 

\begin{definition}\label{def:uf}
A utility function $\mathbf{u} \colon \mathcal{A}^\ast \times \mathbb{R} \to \mathbb{R}$ is a family of functions $(u_S)_{S\in \mathcal{A}^\ast}$ such that $u_S \colon \mathbb{R} \to \mathbb{R}$ is strictly monotone increasing, continuous, and its domain is $\mathbb{R}$. Moreover, the ranges of $u_S$ and $u_T$ are the same for every $S, T \in \mathcal{A}^\ast$; let $R_{\mathbf{u}}$ denote the common range.
\end{definition}

Let the $\mathbf{u}$-excess of a coalition $S\in \mathcal{A}$ by the payoff vector $x\in \mathbb{R}$ in the game $v$ be as follows: $u_S \circ e(S, x) = u_S (v(S)-x(S))$. Moreover, let the $\mathbf{u}$-excess vector be defined as $E(x) := (u_S (e (S,x)))_{S \in \mathcal{A}^\ast} \in \mathbb{R}^{|\mathcal{A}^\ast|}$, where $E (x)_i \geq E(x)_j$ if $i \leq j$.


We can now define the $\mathbf{u}$-prenucleolus similarly to the percapita prenucleolus.

\begin{definition}\label{def:upren}
The $\mathbf{u}$-prenucleolus is the set of preimputations, which lexicographically minimizes the $\mathbf{u}$-excess vectors over the set of preimputations. Formally,

\begin{equation*}
N^\ast_{\mathbf{u}} (v) :=\{ x \in I^*(v) \colon E_{\mathbf{u}} (x) \leq _L E_{\mathbf{u}} (y) \ \forall y\in I^*(v) \}.
\end{equation*}
\end{definition}

\begin{example}
Some examples of utility functions:
\begin{itemize}
\item If $\mathbf{u}$ is the identity function, then the $\mathbf{u}$-penucleolius is the prenucleolus.

\item If $\mathbf{u}$ is defined for all $S\in \mathcal{A}^*$ as $u_S(t) = \frac{t}{|S|}$, then the $\mathbf{u}$-prenucleolus is the percapita prenucleolus.

\item We can also define $\mathbf{u}$ as a shift by a constant $c$. In this case $u_S(t) = t+c$, and for any game $v\in \mathcal{G}^{N, \mathcal{A}}$ the $\mathbf{u}$-prenucleolus is the prenucleolus of the game $v'$, where $v'(S) = v(S) + c$ for all $S\in \mathcal{A}^\ast$, and $v' (N) = v(N)$. Since the prenucleolus is invariant for shifting, in this case the prenucleolus and the $\mathbf{u}$-prenucleolus of a game coincide.

\item Note that $\mathbf{u}$ is not necessarily a family of linear functions. For example $u_S(t) = \arctan(t)$ for all $S\in \mathcal{A}^*$ can also be a utility function.
\end{itemize}
\end{example}

Next, we introduce a generalization of the core \citep{Shapley1955,Gillies1959}: 

\begin{definition}\label{def:ucore}
Given a utility function $\mathbf{u}$, the $\ucore$ of a game $v \in \mathcal{G}^{N,\mathcal{A}}$ is defined as follows:

\begin{equation*}
\ucore (v) := \left\{ x\in \mathbb{R}^{N} \colon x (N) = v(N) \textup{ and } u_S \circ e(S, x) \leq 0, \forall S \in \mathcal{A}^\ast \right\} .
\end{equation*}
\end{definition}

Notice that, if $\mathcal{A} = \mathcal{P} (N)$ and $\mathbf{u}$ is the identity function, then the $\mathbf{u}$-core is the core.

\section{$\mathbf{u}$-balanced games}\label{sec:balanced}
 
Let $\mathfrak{B}$ denote the class of balanced set systems of $\mathcal{A}$. 
 
\begin{definition}
Given a game $v \in \mathcal{G}^{N,\mathcal{A}}$ and a utility function $\mathbf{u}$, the game $v$ is $\mathbf{u}$-balanced, if either $R_{\mathbf{u}} \subseteq \mathbb{R}_- \setminus \{0\}$ or if $0\in R_{\mathbf{u}}$ and

\begin{equation}\label{eq:balancedGame}
\max \limits_{\mathcal{B} \in \mathfrak{B}} \left(\lambda_N v(N) + \sum \limits_{S \in \mathcal{B} \setminus \{ N\} } \lambda_S \, (v (S) - \mathbf{u}^{-1}_S (0)) \right) \leq v (N) ,
\end{equation}

\noindent where $(\lambda_S)_{S \in \mathcal{B}}$ is the balancing weight system of the balanced coalition system $\mathcal{B}$. 
\end{definition}

Notice that, if $\mathcal{A} = \mathcal{P} (N)$ and $\mathbf{u}$ is the identity function, then the $\mathbf{u}$-balancedness reverts to balancedness.

The following theorem is a generalization of the Bondareva--Shapley theorem \citep{Bondareva1963,Shapley1967,Faigle1989} to games with utilities.

\begin{theorem}
Given a game $v \in \mathcal{G}^{N,\mathcal{A}}$ and a utility function $\mathbf{u}$, the $\ucore (v) \neq \emptyset$ if and only if $v$ is $\mathbf{u}$-balanced.
\end{theorem}

\begin{proof}
We consider three cases:

Case 1: $R_{\mathbf{u}} \subseteq \mathbb{R}_{-}\setminus \{ 0\} $. In case of such a utility function, the $\mathbf{u}$-core of a game is always non-empty. The reason for this is that $u_S\circ e(S, x) < 0$ for all $x\in I^*(v)$ and $S\in \mathcal{A}^*$. Therefore, the $\mathbf{u}$-core of the game is not empty, even more, $\mathbf{u}$-core$(v) = I^\ast (v)$.

\bigskip

Case 2: $R_{\mathbf{u}} \subseteq \mathbb{R}_{+}\setminus \{ 0\} $. In case of such a utility function the $\mathbf{u}$-core of a game is always empty. The reason behind it is that $u_S\circ e(S, x) > 0$ for all $x\in I^*(v)$ and $S\in \mathcal{A}^*$; therefore, the $\mathbf{u}$-core of the game is empty.

\bigskip

Case 3: Otherwise, that is, $0\in R_{\mathbf{u}}$. First, consider the following problem

\begin{equation}\label{P1}
\begin{array}{lrl}
& x (N) & \to \min \\
\textup{s.t.} & u_S \circ e (S,x) & \leq 0 \mspace{50mu} \forall S \in \mathcal{A}^* \\
& e(N, x) & \leq 0 \\
& x & \in \mathbb{R}^N \\   
\end{array}
\end{equation}

The problem \eqref{P1} is equivalent to the following LP (here we use that $0\in R_u$)

\begin{equation}\label{P2}
\begin{array}{lrl}
& x (N) & \to \min \\
\textup{s.t.} &  e (S,x) & \leq u_S^{-1}(0) \mspace{50mu} \forall S \in \mathcal{A}^* \\
& e(N, x) & \leq 0 \\
& x & \in \mathbb{R}^N \\   
\end{array}
\end{equation}

The LP \eqref{P2} is equivalent to the following LP

\begin{equation}\label{P3}
\begin{array}{lrl}
& x (N) & \to \min \\
\textup{s.t.} &  x(S) & \geq v(S) - u_S^{-1}(0) \mspace{50mu} \forall S \in \mathcal{A}^*  \\
& x(N) & \geq v(N) \\
& x & \in \mathbb{R}^N \\   
\end{array}
\end{equation}

It is easy to see that LP \eqref{P3} always has a feasible solution. Moreover, since its objective function is bounded from below ($x(N)\geq v(N)$) it always has an optimal solution. Let $x^*$ denote the optimal solution of \eqref{P3}. Then the $\mathbf{u}$-core is nonempty if and only if $x^*(N) = v(N)$.

The dual of \eqref{P3} is the following:

\begin{equation*}
\begin{array}{lrl}
& \lambda_N v(N) + \sum_{S\in \mathcal{A}^* } \lambda _S (v(S)-u_S^{-1}(0)) & \to \max \\
\textup{s.t.} &  \sum_{S\in \mathcal{A}^* \cup \{N\}} \lambda_S \chi_S & = \chi_N  \\
& \lambda_S & \geq 0 \mspace{50mu} \forall S \in \mathcal{A}^*\cup \{N\} \\   
\end{array}
\end{equation*}

By the strong duality theorem of LPs we know that the optimum of the primal LP equals the optimum of the dual LP.

Suppose that $x^*$ and $\lambda ^*$ are optimal solutions of the primal and the dual LPs, respectively. Notice, that $\lambda^*_N v(N) + \sum_{S\in \mathcal{A}^*} \lambda^* _S (v(S)-u_S^{-1}(0))$ equals the left hand side of \eqref{eq:balancedGame}. Due to the strong duality theorem, it is less or equal than $v(N)$ if and only if $x^*(N)$ is less than or equal to $v(N)$, which is equivalent to the $\mathbf{u}$-core being nonempty.
\end{proof}

\section{A lexicographic center approach for the $\mathbf{u}$-prenucleolus}\label{sec:lexi}

In this section, we introduce a modification of the lexicographic center algorithm \citep{Kopelowitz1967,MaschlerPelegShapley1979} for the $\mathbf{u}$-prenucleolus. More precisely, we show how the idea behind the lexicographic center algorithm can be applied for the $\mathbf{u}$-prenucleolus.

The lexicographic center algorithm works by solving a series of LPs. Note that in our case the optimization problems are not necessarily linear. We do not provide any algorithm to solve the non-linear problems, but we introduce a condition to decide whether the problems have optimums or not. 
The following lemma provides the considered condition:

\begin{lemma}\label{lemma:lin2u}
Let $v \in \mathcal{G}^{N, \mathcal{A}}$ be a game, $\mathbf{u}$ be a utility function and $X \subseteq I^*(v)$. Then

\begin{equation}\label{eq:lemma_lin}
\begin{array}{llr}
& k  \to \min & \\
\textup{s.t.} & e(S, x) \leq k & S\in \mathcal{A}^* \\
& x \in X & \\
\end{array}
\end{equation}
has an optimal solution, if and only if
\begin{equation}\label{eq:lemma_u_1}
\begin{array}{llr}
& k \to \min & \\
\textup{s.t.} & u_S \circ e(S, x) \leq k & S\in \mathcal{A}^* \\
& x \in X & \\
\end{array}
\end{equation}
has an optimal solution.
\end{lemma}

\begin{proof}
If $X = \emptyset$, then neither problem \eqref{eq:lemma_lin} nor problem \eqref{eq:lemma_u_1} has an optimal solution.
Therefore, w.l.o.g. we can assume that $X \neq \emptyset$.

Since $X$ is nonempty, if problem \eqref{eq:lemma_lin} does not have an optimal solution, then for every $k \in \mathbb{R}$ there exists an $x_k \in X$ such that $\max_{S\in \mathcal{A}^\ast} e(S, x_k) \leq k$.

Let us define the following sequence: let $k_1\in \mathbb{R}$ be an arbitrary number and $x_1\in X$ a payoff vector such that $\max_{S\in \mathcal{A}^\ast} e(S, x_1) \leq k_1$.

Let $k_2:= \min_{S\in \mathcal{A}^\ast} e(S, x_1) -1$, and $x_2$ be such that $\max_{S\in \mathcal{A}^\ast} e(S, x_2) \leq k_2$.

For $i >2$ let $k_i:= \min_{S\in \mathcal{A}^\ast} e(S, x_{i-1}) -1$, and $x_i\in X$ be such that $\max_{S\in \mathcal{A}^*} e(S, x_i) \leq k_i$.

Then for every $n \in \mathbb{N}^+$ $\min_{S\in \mathcal{A}^\ast} e(S, x_n) > \max_{S\in \mathcal{A}^\ast} e(S, x_{n+1})$, therefore, for all $S\in \mathcal{A}^*$ we have that $e(S, x_n) > e(S, x_{n+1})$. Since $u_S$ is strictly monotone increasing for every $S\in \mathcal{A}^*$, it follows that for every $S\in \mathcal{A}^*$ it holds that $u_S \circ e(S, x_n) > u_S \circ e(S, x_{n+1})$. Therefore, $\max_{S\in \mathcal{A}^*} u_S \circ e(S, x_n) > \max_{S\in \mathcal{A}^*} u_S \circ e(S, x_{n+1})$.

So for every $x\in X$ there exists an $x' \in X$ such that 
\begin{equation*}
\max_{S\in \mathcal{A}^*} u_S \circ e(S, x) > \max_{S\in \mathcal{A}^*} u_S \circ e(S, x').
\end{equation*}
Therefore, \eqref{eq:lemma_u_1} does not have an optimal solution.

\bigskip

Now suppose that \eqref{eq:lemma_u_1} does not have an optimal solution. Since $X$ is nonempty and $D_{u_S} = \mathbb{R}$ (where $D_f$ is the domain of $f$) for every $S\in \mathcal{A}^*$ we have that problem \eqref{eq:lemma_u_1} has a feasible solution. Let $R_{\mathbf{u}} = (a, b)$ denote the range of $\mathbf{u}$, where $a, b \in \mathbb{R}\cup \{ \infty, -\infty \} $.

If problem \eqref{eq:lemma_u_1} does not have an optimal solution, then $\inf\{ t \colon u_S \circ e(S, x) \leq t \ \forall S\in \mathcal{A}^*, x\in X\} = a$. 
It means that for every $k \in (a, b) \cap \mathbb{R}$ there exists $x_k \in X$ such that $\max_{S\in \mathcal{A}^\ast} u_S \circ e(S, x_k) \leq k$.
Furthermore, it is clear that $\min_{S\in \mathcal{A}^\ast} u_S \circ e(S, x) \in (a, b) \cap \mathbb{R}$ for every $x\in X$.

Let us define the following sequence:
let $k_1 \in (a, b) \cap \mathbb{R}$ be an arbitrary number and $x_1\in X$ be such that $\max_{S\in \mathcal{A}^*} u_S \circ e(S, x_1) \leq k_1$.

Since $\min_{S\in \mathcal{A}^*} u_S \circ e(S, x_1) \in (a, b)$, there exists $\varepsilon > 0$ such that $\min_{S\in \mathcal{A}^*} u_S \circ e(S, x_1) - \varepsilon > a$.
Let $k_2 := \min_{S\in \mathcal{A}^*} u_S \circ e(S, x_1) - \varepsilon$ and $x_2 \in X$ be such that $\max_{S\in \mathcal{A}^*} u_S \circ e(S, x_2) \leq k_2$.

For any $n > 2$: let $\varepsilon_n >0$ be such that $\min_{S\in \mathcal{A}^*} u_S \circ e(S, x_{n-1}) - \varepsilon_n > a$. Let $k_n := \min_{S\in \mathcal{A}^\ast} u_S \circ e(S, x_{n -1}) - \varepsilon_n$ and $x_n \in X$ be such that $\max_{S\in \mathcal{A}^\ast} u_S \circ e(S, x_n) \leq k_n$.

Then, for every $n \in \mathbb{N}^+$ $\min_{S\in \mathcal{A}^\ast} u_S \circ e(S, x_n) > \max_{S\in \mathcal{A}^\ast} u_S \circ e(S, x_{n+1})$, furthermore, for each $S\in \mathcal{A}^\ast$ it holds that $u_S \circ e(S, x_n) > u_S \circ e(S, x_{n+1})$. Since $u_S$ is strictly monotone increasing, it holds that $e(S, x_n) > e(S, x_{n+1})$,  for all $S\in \mathcal{A}^\ast$. Meaning that $\max_{S\in \mathcal{A}^\ast} e(S, x_n) > \max_{S\in \mathcal{A}^\ast} e(S, x_{n+1})$.

So, for every $x\in X$ there exists an $x'\in X$ such that
\begin{equation*}
\max_{S\in \mathcal{A}^*} e(S, x) > \max_{S\in \mathcal{A}^*} e(S, x') \, .
\end{equation*}
Therefore, problem \eqref{eq:lemma_lin} does not have an optimal solution either.
\end{proof}

The following lemma is a direct corollary of Lemma \ref{lemma:lin2u}.

\begin{lemma}\label{lemma:lexi}
Let $v \in \mathcal{G}^{N,\mathcal{A}}$ be a game, and $\mathbf{u^1}$, $\mathbf{u^2}$ be utility functions. Let $X\subseteq I^*(v)$, then 
\begin{equation*}
\min_{x\in X} \max_{S\in \mathcal{A}^*} u^1_S \circ (v(S)-x(S))
\end{equation*}
exists if and only if
\begin{equation*}
\min_{x\in X} \max_{S\in \mathcal{A}^*} u^2_S \circ (v(S)-x(S))
\end{equation*}
exists.
\end{lemma}

Next, we introduce a variant of the lexicographic center algorithm \citep{Kopelowitz1967,MaschlerPelegShapley1979}, which can be used for calculating the $\mathbf{u}$-prenucleolus of a game. 

Let $v \in \mathcal{G}^{N,\mathcal{A}}$ be a game and $\mathbf{u}$ be a utility function. Consider the following problem:

\begin{equation}\label{LP1}
\begin{array}{llr}
& t \to \min & \\
\textup{s.t.} & u_S \circ e(S, x) \leq t, & \ S \in \mathcal{A}^\ast \\
& x \in I^\ast (v) &  \\
& t \in R_{\mathbf{u}} & 
\end{array}
\end{equation}

If problem \eqref{LP1} has an optimal solution, let $t_1$ denote the optimum of \eqref{LP1}.

Let $X_1$ be defined as follows:

\begin{equation*}
X_1 = \{ x \in I^\ast (v) \colon u_S \circ e(S,x) \leq t_1, \  \forall S \in \mathcal{A}^\ast \} .
\end{equation*}

\noindent Furthermore, let

\begin{equation*}
W_1 = \{ S \in \mathcal{A}^\ast \colon \exists c_S \in \mathbb{R}, \text{ such that } u_S \circ e(S,x) = c_S,\ \forall x \in X_1 \} .
\end{equation*}

Let $k \geq 2$, and let us consider the following problem:

\begin{equation}\label{LPk}
\begin{array}{lll}
& t \to \min & \\
\textup{s.t.} & u_S \circ e(S,x) \leq t , & \ S \in \mathcal{A}^\ast \setminus (\cup_{r=1}^{k-1} W_r) \\
& x \in X_{k-1} &  \\
& t \in \mathbb{R} & 
\end{array}
\end{equation}

If \eqref{LPk} has an optimal solution, let $t_k$ denote the optimum of \eqref{LPk}. 

Let $X_k$ be defined as follows

\begin{equation*}
X_k = \{ x \in X_{k-1} \colon u_S \circ e(S,x) \leq t_k, \  \forall S \in \mathcal{A}^\ast \setminus (\cup_{r=1}^{k-1} W_r) \} .
\end{equation*}

\noindent Furthermore, let

\begin{equation*}
W_k = \{ S \in \mathcal{A}^\ast \colon \exists c_S \in \mathbb{R}, \text{ such that } u_S \circ e(S,x) = c_S,\ \forall x \in X_k \} .
\end{equation*}

It is easy to see that $t_k \geq t_{k+1}$ and $X_k \supseteq X_{k+1}$ for all $k \in \mathbb{N}_+$, and there exists $k^*$ such that for all $l\geq k^*$ it holds that $X_l = X_{k^*}$, and $X_{k^*} \neq \emptyset$.

\bigskip

\citet{MaschlerPottersTijs1992} proved that a more general version of the above algorithm - the lexicographical center algorithm for finding the general prenucleolus - returns with the general prenucleolus. The $\mathbf{u}$-prenucleolus is a special case of the general prenucleolus, hence the result by \citet{MaschlerPottersTijs1992} implies the following theorem:

\begin{theorem}\label{thm:charupren}
For every game $v \in \mathcal{G}^{N,\mathcal{A}}$ and utility function $\mathbf{u}$ it holds that

\begin{equation*}
N_{\mathbf{u}}^\ast (v) = X_{k^\ast}.
\end{equation*}
\end{theorem}


\section{The nonemptyness of the $\mathbf{u}$-prenucleolus}\label{sec:existence}

In case of classical TU-games ($\mathcal{A} = \mathcal{P}(N)$ and $u_S$ is the identity function for all $S \in \mathcal{A}^*$), the prenucleolus always consists of exactly one point (payoff vector) \citep{Schmeidler1967}. However, this does not hold in case of TU-games with restricted cooperation, where the prenucleolus is a set -- not necessarily a singleton set -- of payoff vectors. \citet{Katsev2013} showed that the prenucleolus of a game is nonempty if and only if the set of feasible coalitions is a balanced set of coalitions. In this section we are going to prove that this statement holds for the $\mathbf{u}$-prenucleolus of a game as well. The proof relies on the generalization of Kohlberg's theorem (Theorem \ref{th:Kohlberg}) and on Lemma \ref{lemma:lexi}.

\citeauthor{Kohlberg1971}'s theorem (\cite{Kohlberg1971}) for classical TU-games is as follows: 

\begin{theorem}[\citeauthor{Kohlberg1971}'s theorem]
Given a game $v \in \mathcal{G}^N$ and a payoff vector $x\in I^*(v)$, $x$ is the prenucleous if and only if for every $\alpha \in \mathbb{R}$ it holds that
either $\mathcal{D}(\alpha, x^\ast) := \{ S\in \mathcal{P}^*(N) \colon e(S, x) \geq \alpha \} = \emptyset$, or $\mathcal{D}(\alpha, x^\ast)$ is a balanced set of coalitions.
\end{theorem}

\citet{MaschlerPottersTijs1992} proposed a generalisation of \citeauthor{Kohlberg1971}'s theorem for the general prencucleolus. The $\mathbf{u}$-prenucleolus is a special case of the general prenucleolus, therefore the following generalization of Kohlberg's theorem is a corollary of the result by \cite{MaschlerPottersTijs1992}:

\begin{theorem}[Generalization of \citeauthor{Kohlberg1971}'s theorem]\label{th:Kohlberg}
Given a game $v \in \mathcal{G}^{N,\mathcal{A}}$, a utility function $\mathbf{u}$, and $x \in I^\ast (v)$: $x$ is an element of the $\mathbf{u}$-prenucleolus ($x \in N_{\mathbf{u}}^\ast (v)$) if and only if $\mathcal{D}_{\mathbf{u}}(\alpha, x)$ is a balanced set of coalitions for every $\alpha$ such that $\mathcal{D}_{\mathbf{u}}(\alpha, x) \neq \emptyset$, where $\mathcal{D}_{\mathbf{u}}(\alpha , x^\ast) := \{ S\in \mathcal{A}^\ast \colon  u_S \circ e(S, x^\ast) \geq \alpha \}$.
\end{theorem}

The following proposition, which generalizes Theorem 2. on page 58 of \citet{Katsev2013}, is a corollary of Theorem \ref{th:Kohlberg}.

\begin{proposition}\label{th:Katsev1}
Let $v\in \mathcal{G}^{N, \mathcal{A}}$ be a game and $\mathbf{u}$ be a utility function. If the $\mathbf{u}$-prenucleolus of the game is nonempty, then $\mathcal{A}^\ast$ is a balanced set of coalitions.
\end{proposition}

\begin{proof}
Since the $\mathbf{u}$-prenucleolus is nonempty, there exists $x \in N_{\mathbf{u}}^\ast (v)$. Then by Theorem \ref{th:Kohlberg} it holds that $\mathcal{D}_{\mathbf{u}}(\alpha , x)$ is a balanced set of coalitions for every $\alpha$ such that $\mathcal{D}_{\mathbf{u}}(\alpha, x) \neq \emptyset$.

Take $\alpha^\ast$ such that it is the smallest element of $E_{\mathbf{u}}(x)$, that is, let $\alpha^* := E_{\mathbf{u}}(x)_{|\mathcal{A}^\ast|}$. Then, $\mathcal{D}_{\mathbf{u}}(\alpha ^\ast, x) = \mathcal{A}^\ast$, hence, by Theorem \ref{th:Kohlberg}, $\mathcal{A}^\ast$ is a balanced set of coalitions.
\end{proof}

\begin{proposition}\label{th:Katsev2}
Let $v\in \mathcal{G}^{N, \mathcal{A}}$ be a game and $\mathbf{u}$ be a utility function. If $\mathcal{A}^*$ is a balanced set of coalitions, then the $\mathbf{u}$-prenucleolus of the game is nonempty.
\end{proposition}

\begin{proof}
The $\mathbf{u}$-prenucleolus of the game is nonempty, if the considered optimization problem attains an optimal solution in every iteration of the generalized lexicographic center approach discussed in Section \ref{sec:lexi}. We are going to show that if $\mathcal{A}^*$ is a balanced set of coalitions, then in every iteration of the generalized lexicographic center approach the considered optimization problem attains an optimal solution.

Take an arbitrary step in the generalized lexicographic center approach (see section \ref{sec:lexi}). Then, we have to solve the following minimization problem:

\begin{equation}\label{eq:Katsev_u}
\begin{array}{llr}
& t \to \min & \\
\textup{s.t.} & u_S \circ e(S, x) \leq t, & \ S \in \mathcal{A}^\ast \setminus W \\
& u_S \circ e(S, x) = c_S, & S\in W \\
& x\in I^*(v) &  \\
& t \in R_{\mathbf{u}}, & 
\end{array}
\end{equation}
where $W=\cup_{i=1}^k W^u_i$ for some $k$.

By Lemma \ref{lemma:lexi}, problem \eqref{eq:Katsev_u} has an optimal solution if and only if the following LP has:

\begin{equation}\label{LPallitas}
\begin{array}{llr}
& t \to \min & \\
\textup{s.t.} & v(S)-x(S) \leq t, & \ S \in \mathcal{A}^\ast \setminus W \\
& v(S) - x(S) = u_S^{-1}(c_S), & S\in W \\
& x(N) = v(N) &  \\
& t \in R_{\mathbf{u}} & 
\end{array}
\end{equation}

Since the LP \eqref{LPallitas} has a feasible solution (every solution of the optimization problem in the previous iteration is a feasible solution here, or, if we are at the first iteration, that is $W= \emptyset$, then it is easy to see that LP \eqref{LPallitas} has a feasible solution), it is enough to prove that the objective function in \eqref{LPallitas} is bounded from below. 

Let $\{ \lambda_S\} _{S\in \mathcal{A}^\ast}$ be a balancing weight system, then 

\begin{eqnarray*}
\sum_{S\in \mathcal{A}^\ast \setminus W} \lambda_S (v(S)-t) + \sum_{S\in W} \lambda_S (v(S)-u_S^{-1}(c_S)) & \leq & \sum_{S\in \mathcal{A}^\ast} \lambda_S \chi_S^\top = v(N) \\
\sum_{S\in \mathcal{A}^\ast \setminus W} \lambda_S v(S) + \sum_{S\in W} \lambda_S (v(S)- u_S^{-1}(c_S)) - v(N) & \leq & \sum_{S\in \mathcal{A}^\ast \setminus W} \lambda_S t, \\
\end{eqnarray*}

\noindent where $\chi_S$ denotes the characteristic vector of a set $S\subseteq N$.

Since $0<\lambda_S$, $\forall S\in \mathcal{A}^\ast$ it holds that $\sum_{S\in \mathcal{A}^\ast \setminus W} \lambda_S >0$, hence

\begin{equation*}
\frac{\sum_{S\in \mathcal{A}^* \setminus W} \lambda_S v(S) + \sum_{S\in W} \lambda_S (v(S)- u_S^{-1}(c_S)) - v(N)}{\sum_{S\in \mathcal{A}^* \setminus W} \lambda_S} \leq t
\end{equation*}

The left side of the inequality is a constant, hence it gives a lower bound for the right side, which is the objective function of problem \eqref{LPallitas}. Therefore, it is bounded from below meaning that problem \eqref{LPallitas} has an optimal solution. 
\end{proof}

The following theorem, which generalizes Theorem 2. on page 58 of  \citet{Katsev2013} comes from Propositions \ref{th:Katsev1} and \ref{th:Katsev2}:

\begin{theorem}\label{th:Katsew}
Let $v\in \mathcal{G}^{N, \mathcal{A}}$ be a game and $\mathbf{u}$ be a utility function. Then $\mathcal{A}^\ast$ is a balanced set of coalitions, if and only if the $\mathbf{u}$-prenucleolus of the game is nonempty.
\end{theorem}

\section{The cardinality of the $\mathbf{u}$-prenucleolus}\label{sec:unicity}


In this section, we consider the size of the $\mathbf{u}$-prenucleolus. First, we introduce a notation. For a family of coalitions $\mathcal{A} \subseteq \mathcal{P} (N)$ let $X(\mathcal{A})$ denote the $|\mathcal{A}| \times |N|$ dimensional matrix, where its row vectors are the characteristic vectors of the sets from $\mathcal{A}$.

The following theorem is a generalization of Theorem 3. on page 59 of \citet{Katsev2013}.

\begin{theorem}
Given a game $v \in \mathcal{G}^{N,\mathcal{A}}$, where $\mathcal{A}$ is a balanced set of coalitions, and a utility function $\mathbf{u}$, the $\mathbf{u}$-prenucleolus of the game $v$ is a singleton if and only if rank$(X(\mathcal{A})) = |N|$.
\end{theorem}

\begin{proof}
Only if: Let $x \in N_{\mathbf{u}}^\ast (v)$ be the only element of the $\mathbf{u}$-prenucleolus. Suppose for contradiction that rank$(X(\mathcal{A})) < |N|$. Consider the following system of linear equations:

\begin{equation}\label{eq:Katsev_onlyif}
\begin{cases}
y(S)  = x(S), \ \forall S\in \mathcal{A}^* \\
y(N)  = v(N)
\end{cases}
\end{equation}

\noindent  Then \eqref{eq:Katsev_onlyif} can be rewritten as follows:

\begin{equation*}
X(\mathcal{A})y = e(x),
\end{equation*}

\noindent where $e(x) = (0,\dots , x(S), \dots, v(N))^\top$, $S\in \mathcal{A}$.

Since 
 rank$(X(\mathcal{A})) < |N|$, the system \eqref{eq:Katsev_onlyif} has an infinite many solutions and all of those belong to the $\mathbf{u}$-prenucleolus, which is a contradiction.
 
\bigskip

If: Suppose for contradiction that rank$(X(\mathcal{A})) = |N|$ and there exist $x, y \in N_{\mathbf{u}}^\ast(v)$ such that $x \neq y$. This means that $E_{\mathbf{u}}(x) = E_{\mathbf{u}}(y)$.

Notice that for all $S\in \mathcal{A}^\ast$, such that $x(S) = y(S)$ we have the following:

\begin{equation*}
u_S \circ \left(v(S) - \frac{x+y}{2}(S)\right) = u_S\circ (v(S)-x(S)) =  u_S\circ (v(S)-y(S)).
\end{equation*}

If there exists a coalition $S\in \mathcal{A}^*$ such that $x(S) \neq y(S)$ (without loss of generality we can suppose that $x(S) > y(S)$), then 

\begin{equation*}
u_S\circ (v(S)-x(S)) < u_S \circ \left(v(S) - \frac{x+y}{2}(S)\right) < u_S\circ (v(S)-y(S)).
\end{equation*}

Then let $T_i$ be the first coalition  according to the order by vector $E_v^{\mathbf{u}} (\frac{x+y}{2})$ for which $x(T_i) \neq y(T_i)$. Then either  $u_{T_i} \circ (v(T_i) - \frac{x+y}{2}(T_i)) < u_{T_i}\circ (v(T_i)-y(T_i))$ or $u_{T_i} \circ (v(T_i) - \frac{x+y}{2}(T_i)) < u_{T_i}\circ (v(T_i)-x(T_i))$.

Since the $\mathbf{u}$-excesses are in non-increasing order in $E_v^\mathbf{u}$ we have that $E_v^{\mathbf{u}}(\frac{x+y}{2}) <_L E_v^{\mathbf{u}}(x) = E_v^{\mathbf{u}}(y)$. Therefore, $x(S) = y(S)$ for all $S\in \mathcal{A}$, and $y$ is a solution of the following system of linear equations:

\begin{equation}
\begin{cases}
z(S)  = x(S),  \forall S\in \mathcal{A}^* \\
z(N)  = v(N)
\end{cases}
\end{equation}

\noindent which can be rewritten as

\begin{equation*}
X(\mathcal{A})z = e(x).
\end{equation*}

Since $rank(X(\mathcal{A})) = |N|$, this system has a unique solution $y=x$, which is a contradiction.
\end{proof}

\section{The $\mathbf{u}$-essential coalitions}\label{sec:essential}

\citet{Huberman1980} showed that the so-called essential coalitions give a characterization set for the nucleolus of balanced TU-games. Since in case of balanced games, the nucleolus and the prenucleolus coincide, the essential coalitions also give a characterization set for the prenucleolus. First, consider the definition of essential coalitions used by \citet{Huberman1980}. 

\begin{definition}\label{def:ess}
Let $v \in \mathcal{G}^N$ be a game. Then, a coalition $S \in \mathcal{P}^\ast (N)$ is \textit{essential}, if either $|S| = 1$, or 

\begin{equation*}
v(S) > \max \limits_{\mathcal{B} \in \mathcal{D}_S} \sum_{T \in \mathcal{B}} v(T) .
\end{equation*}

Let $\mathcal{E}_v$ denote the class of essential coalitions of the $v$ game.
\end{definition}

Here is \citet{Huberman1980}'s theorem:

\begin{theorem}[\cite{Huberman1980}]\label{th:Huberman}
Let $v \in \mathcal{G}^N$ be a balanced game. Then  $ \mathcal{E}_v$ is a characterization set for the nucleolus, that is, the values $(v (S))_{S \in \mathcal{E}_v}$ determine the nucleolus of the game $v$.
\end{theorem}


When generalizing Huberman's theorem, we need to "redefine" the essential coalitions. 

\begin{definition}\label{def:u-ess}
Given a game $v \in \mathcal{G}^{N, \mathcal{A}}$ and utility function $\mathbf{u}$, a coalition $S \in \mathcal{A}^{\ast}$ is \textit{$\mathbf{u}$-essential}, if either $\mathcal{D}_S^{\mathcal{A}^{\ast}} = \emptyset $ or if $\exists x\in \mathbf{u}\text{-core}(v)$ such that

\begin{equation*}
u_S \circ e(S, x) > \max \limits_{\mathcal{B} \in \mathcal{D}_S^{\mathcal{A}^{\ast}}} \sum_{T \in \mathcal{B}} u_T \circ e(T, x) .
\end{equation*}

Let $\mathcal{E}_v^{\mathbf{u}}$ denote the class of $\mathbf{u}$-essential coalitions of the game $v$.
\end{definition}

Note, that Definition \ref{def:u-ess} uses $x$, while Definition \ref{def:ess} seemingly does not. However, if $\mathbf{u}$ is the identity function, $x$ is cancelled out from the inequality. Moreover, the idea of the proof of Huberman's theorem is that the excesses of the essential coalitions exceed the excesses of the other coalitions. 

Notice, that if $\mathbf{u}$ is the identity function and $\mathcal{A} = \mathcal{P}(N)$; then, the $\mathbf{u}$-essential coalitions are the essential coalitions, hence the $\mathbf{u}$-essential coalitions are generalizations of the essential coalitions.

We should also note, that the inclusion of $x$ in Definition \ref{def:u-ess} raises the question which $x$ must be considered. A straightforward option would be that every element of the preimputations, as it is stated in the definition of the $\mathbf{u}$-prenucleolus, have to be considered. However, it turns out that not all preimputations must be considered, in Definition \ref{def:u-ess} we choose the elements of the $\mathbf{u}$-core. What is  more, even a smaller set, the $\mathbf{u}$-least-core could be applied in the definition of $\mathbf{u}$-essential coalitions. 







\begin{example}
Consider the following game, which is a modification of the game examined in Example 3 in \citet{Solymosi2019}: $v(N) = 12, v(\{ 1, 2\} ) = v(\{ 3, 4\} ) = v(\{ 2, 3, 4\} ) = 6, v(\{ 1, 4\} ) = 4, v(\{ 4\} ) = 3, v(\{ 1, 2, 3\} ) = 9$ and for every other coalition $S\in \mathcal{P}(N)$ let $v(S) = 0$.

Let the utility function be the percapita-utility function, that is $u_S(t) = \frac{t}{|S|}$ for all $S\in \mathcal{P}^*(N)$.

The percapita-core coincides with the core, so $\mathbf{u}-$core$(v) = \{ (t, 6-t, 3, 3) \colon 1\leq t \leq 6 \}$.

The percapita-prenucleolus of $v$ is $(3, 3, 3, 3)$. The first iteration in the lexicographic center algorithm (see \eqref{LP1}) gives $t_1 = 0$ and 

\begin{equation*}
W_1 = \{ \{ 3\} ,  \{ 4\} , \{ 1, 2\} ,  \{ 3, 4\} , \{ 1, 2, 3\} ,  \{ 1, 2, 4\} \}.
\end{equation*}

The second iteration (see \eqref{LPk}) gives $t_2 = -1$ and  

\begin{equation*}
W_2 = \{  \{ 1\} ,  \{ 2\} ,  \{ 1, 3\} ,  \{ 1, 4\} ,  \{ 2, 3\} ,  \{ 2, 4\} ,  \{ 1, 3, 4\} ,  \{ 2, 3, 4\} \} .
\end{equation*}

Therefore, the $\mathbf{u}$-prenucleolus is $(3, 3, 3, 3)$.

However, if we consider only the essential coalitions in the calculation of the percapita-prenucleolus of the game, we get $(4 + \frac{1}{3}, 1 + \frac{2}{3}, 3, 3)$. The essential coalitions of $v$ are: 

\begin{equation*}
\mathcal{E}_v = \{ \{ 1 \} ,  \{ 2\} ,  \{ 3\} ,  \{ 4\} ,  \{ 1, 2\} ,  \{ 3, 4\} ,  \{ 1, 4\} ,  \{ 1, 2, 3\} \}.
\end{equation*}

The first iteration gives $t_1 = 0$ and

\begin{equation*}
W_1 = \{  \{ 3\} ,  \{ 4\} ,  \{ 1, 2\} ,  \{ 3, 4\} ,  \{ 1, 2, 3\}  \}.
\end{equation*}

The second iteration gives: $t_2 = -(1 + \frac{2}{3})$ and

\begin{equation*}
W_2 = \{ \{ 1\} , \{ 2\} , \{ 1, 4\} \}.
\end{equation*}

Therefore, the $\mathbf{u}$-prenucleolus of the game with restricted coalitions $\mathcal{E}_v$ is $(4 + \frac{1}{3}, 1 + \frac{2}{3}, 3, 3)$.

However, if we consider the $\mathbf{u}$-essential coalitions when calculating the percapita-prenucleolus, we get $(3, 3, 3, 3)$, which coincides with the percapita-prenucleolus. Indeed, the $\mathbf{u}$-essential coalitions are 

\begin{equation*}
\begin{split}
\mathcal{E}_v^{\mathbf{u}}(v) = \{ \{ 1\} , \{ 2\} , \{ 3\} , \{ 4\} , \{ 1, 2\} , \{ 1, 3\} , \{ 1, 4\} , \{ 2, 3\} , \{ 2, 4\} , \{ 3, 4\} ,\\
 \{ 1, 2, 3\} ,\{ 1, 3, 4\} , \{ 2, 3, 4\}  \} .
\end{split}
\end{equation*}

Notice, that not all coalitions are $\mathbf{u}$-essential.  Coalition $\{ 1, 2, 4\}$ is not $\mathbf{u}$-essential, since for all $x\in \mathbf{u}$-core$(v)$ $u_{\{ 1, 2, 4 \} } \circ e_v(\{ 1, 2, 4\} , x ) = \frac{0-9}{3} = -3$ and $u_{\{ 1, 2\} } \circ e_v(\{ 1, 2\} , x) + u_{\{ 4\} } \circ e_v(\{ 4\} , x) = \frac{6-6}{2} + 3-3 = 0$, hence, $u_{\{ 1, 2, 4 \} } \circ e_v(\{ 1, 2, 4\} , x ) < u_{\{ 1, 2\} } \circ e_v(\{ 1, 2\} , x) + u_{\{ 4\} } \circ e_v(\{ 4\} , x)$.

Considering only the $\mathbf{u}$-essential coalitions the first iteration gives $t_1 = 0$ and 

\begin{equation*}
W_1 = \{ \{ 3\} , \{ 4\} , \{ 1, 2\} , \{ 3, 4\} , \{ 1, 2, 3\} \}.
\end{equation*}

The second iteration gives $t_2 = -1$ and 

\begin{equation*}
W_2 = \{ \{ 1\} , \{ 2\} , \{ 1, 3\} , \{ 1, 4\} , \{ 2, 3\} , \{ 2, 4\} , \{ 1, 3, 4\} , \{ 2, 3, 4\} \}.
\end{equation*}

Therefore, the per-capita prenucleolus of the game with restricted coalitions $\mathcal{E}_v^{\mathbf{u}}$ is $(3, 3, 3, 3)$, which coincides with the percapita-prenucleolus of the original game.
\end{example}

In the following, we show that the $\mathbf{u}$-essential coalitions form a characterization set for the $\mathbf{u}$-prenucleolus in case of $\mathbf{u}$-balanced games. First, consider the following two optimization problems:

\begin{equation}\label{LP100}
\begin{array}{llr}
& t \to \min & \\
\textup{s.t.} & u_S \circ e(S, x) \leq t, & \ S \in \mathcal{A}^\ast \\
& x \in I^\ast (v) &  \\
& t \in R_{\mathbf{u}} & 
\end{array}
\end{equation}

\noindent and

\begin{equation}\label{LP101}
\begin{array}{llr}
& t \to \min & \\
\textup{s.t.} & u_S \circ e(S, x) \leq t, & \ S \in \mathcal{E}_v^\mathbf{u} \\
& x \in I^\ast (v) &  \\
& t \in R_{\mathbf{u}} & 
\end{array}
\end{equation}

Let $t_1$ be the optimum of problem \eqref{LP100} and $X_1$ be the set of optimal solutions of problem \eqref{LP100} except from t, that is, $X_1 = \{x\in I^\ast (v) \colon u_S \circ e_v (S,x) \leq t_1 \ \forall S\in \mathcal{A}^\ast \}$. Similarly, let $t_1'$ be the optimum of problem \eqref{LP101} and $X_1$ be the set of optimal solutions of problem \eqref{LP101} except from $t$, that is, $X_1' = \{x\in I^\ast (v) \colon u_S \circ e_v (S,x) \leq t_1 \ \forall S\in \mathcal{E}^\mathbf{u}_v \}$.
 

Next, we consider some lemmata which are needed for showing that $X_1 = X_1'$ (Proposition \ref{prop:X1=Y1}). Then, by these results we show that the $\mathbf{u}$-essential coalitions characterize the $\mathbf{u}$-prenucleolus of $\mathbf{u}$-balanced games (Theorem \ref{th:u-Huberman}). 

In order to help the reader in following the interdependence of the upcoming results, we have constructed the following graph:  

\bigskip

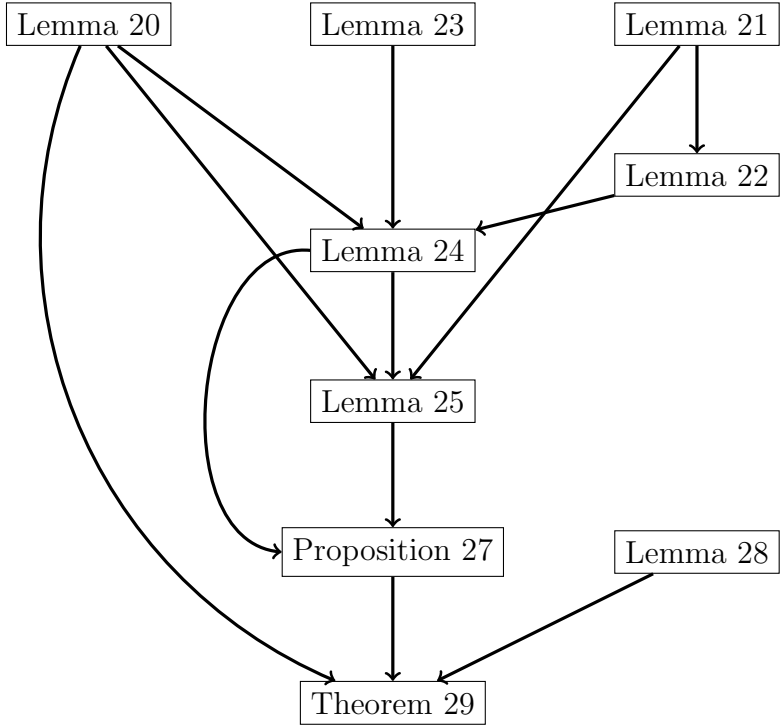
\begin{tikzpicture}
\definecolor{fekete}{rgb}{0,0,0};
\draw (0,0) node[draw] (A) {Lemma \ref{lemma:rest1}};
\draw (4,0) node[draw] (C) {Lemma \ref{lemma:monotone}};
\draw (8,0) node[draw] (B) {Lemma \ref{lemma:X(t')<=X(t")}};
\draw (8,-2) node[draw] (D) {Lemma \ref{lemma:convex}};
\draw[very thick,->] (B) -- (D);
\draw (4,-3) node[draw] (E) {Lemma \ref{lemma:X1=X1'}};
\draw[very thick,->] (C) -- (E);
\draw (4,-5) node[draw] (F) {Lemma \ref{lemma:t<t_1:X1=X1'}};
\draw[very thick,->] (E) -- (F);
\draw (4,-7) node[draw] (G) {Proposition \ref{prop:X1=Y1}};
\draw[very thick,->] (F) -- (G);
\draw (4,-9) node[draw] (I) {Theorem \ref{th:u-Huberman}};
\draw[very thick,->] (G) -- (I);
\draw (8,-7) node[draw] (H) {Lemma \ref{lemma:rest2}};
\draw[very thick,->] (H) -- (I);
\draw[very thick,->] (A) -- (E);
\draw[very thick,->] (A) -- (F);
\draw[very thick,->, bend right=45] (A) to (I);
\draw[very thick,->] (B) -- (F);
\draw[very thick,->] (D) -- (E);
\draw[very thick,->, bend right=90] (E) to (G);

\end{tikzpicture}

\begin{lemma}\label{lemma:rest1}
Let $S \in \mathcal{A}^\ast \setminus \mathcal{E}_v^{\mathbf{u}}$. Then for every $x \in \mathbf{u}\text{-core}(v)$ there exists $\mathcal{B}^\ast \in \mathcal{D}_S^{\mathcal{A}^{\ast}}$ such that $u_S \circ e(S, x) \leq \sum_{T \in \mathcal{B}^\ast} u_T \circ e(T, x)$  and $\mathcal{B}^\ast \subseteq \mathcal{E}_v^{\mathbf{u}}$.
\end{lemma}

\begin{proof}
Since $S$ is not $\mathbf{u}$-essential, $\forall x\in \mathbf{u}\text{-core}(v)$ there exists $\mathcal{B} \in \mathcal{D}_S^{\mathcal{A}^{\ast}}$ such that $u_S \circ e(S, x) \leq \sum_{T \in \mathcal{B}^\ast} u_T \circ e(T, x)$ by definition. For all $x\in \mathbf{u}$-core$(v)$ let $\mathbf{B}(x) := \{ \mathcal{B} \in \mathcal{D}_S^{\mathcal{A}^{\ast}} \colon u_S \circ e(S, x) \leq \sum_{T \in \mathcal{B}} u_S \circ e(S, x) \}$ and let $\mathcal{B}^\ast \in \mathbf{B}(x)$ be such that for every partition $\mathcal{B} \in \mathbf{B}(x)$ it holds that $| \mathcal{B}^\ast | \geq | \mathcal{B} |$.

Indirectly assume that a coalition $T^\ast \in \mathcal{B}^\ast$ is not $\mathbf{u}$-essential. Since $T^*$ is not $\mathbf{u}$-essential by Definition \ref{def:u-ess} $\exists \mathcal{B}' \in\mathcal{D}_{T^\ast}^{\mathcal{A}^{\ast}}$ such that $u_{T^*} \circ e(T^\ast, x)  \leq \sum_{T' \in \mathcal{B}'} u_{T'} \circ e (T', x)$, therefore, 

\begin{equation*}
u_S \circ e(S, x)  \leq \sum_{T' \in (\mathcal{B}^\ast \setminus \{T^\ast\}) \cup \mathcal{B}'} u_{T'} \circ e (T', x),
\end{equation*}

\noindent and $| ( \mathcal{B}^\ast \setminus \{T\}) \cup \mathcal{B}' | > | \mathcal{B}^\ast |$, which is a contradiction.
\end{proof}

Next, we introduce the following notion: for a class of coalitions $\mathcal{S}\subseteq \mathcal{A}^*$ and $t\in \mathbb{R}$ let $X(\mathcal{S}, t):= \{ x\in I^\ast(v) \colon u_S\circ e_v(S, x) \leq t, \forall S\in \mathcal{S} \} $.

\begin{lemma}\label{lemma:X(t')<=X(t")}
Given a game $v\in \mathcal{G}^{N, \mathcal{A}}$ and $t', t'' \in \mathbb{R}$ such that $t'\leq t''$, then $X(\mathcal{S}, t') \subseteq X(\mathcal{S}, t'')$.
\end{lemma}

\begin{proof}
For any $x\in X(\mathcal{S}, t')$ it holds that $x\in I^*(v)$, and for all $S\in \mathcal{S}$:
\begin{equation*}
u_S\circ e_v(S, x) \leq t' \leq t'' \ .
\end{equation*}
Therefore, $x\in X(\mathcal{S}, t'')$.
\end{proof}

\begin{lemma}\label{lemma:convex}
Given a utility function $\mathbf{u}$ and a $\mathbf{u}$-balanced game $v\in \mathcal{G}^{N, \mathcal{A}}$, for every $t_1\leq t$ it holds that both $X(\mathcal{A}^\ast, t)$ and $X(\mathcal{E}_v^{\mathbf{u}}, t)$ are nonempty, convex and closed.
\end{lemma}

\begin{proof}
By Lemma \ref{lemma:X(t')<=X(t")} $X(\mathcal{A}^\ast, t) \supseteq X(\mathcal{A}^\ast, t_1) = \mathbf{u}\text{-least-core}(v)$. We know that $\mathbf{u}\text{-least-core}(v) \neq \emptyset$, and $X(\mathcal{A}^\ast, t) \subseteq X(\mathcal{E}_v^{\mathbf{u}}, t)$, because $\mathcal{E}_v^{\mathbf{u}} \subseteq \mathcal{A}^*$. Therefore, $X(\mathcal{E}_v^\mathbf{u}, t) \neq \emptyset$ as well.

Let $H_S := \{ x\in \mathbb{R}^N \colon u_S \circ e_v (S,x) \leq t \}$ for all $S\in\mathcal{A}^\ast$. Then $H_S = \{ x \in \mathbb{R}^N \colon e_v (S,x) \leq u_S^{-1}(t) \}$, hence $H_S$ is a closed half-space, therefore, it is convex and closed. $X_0 = I^*(v) = \{x\in \mathbb{R}^N \colon x(N) = v(N) \}$ is a hyperplane, therefore it is convex and closed. Finally,

\begin{equation*}
X(\mathcal{A}^*, t) = \{x\in I^*(v) \colon u_S \circ e_v (S,x) \leq t \ \forall S\in \mathcal{A}^* \} = \cap_{S\in \mathcal{A}^*}H_S\cap X_0 \ ,
\end{equation*}

\noindent hence, $X(\mathcal{A}^*, t)$ is an intersection of finitely many convex, closed sets; therefore, it is convex and closed.

Similarly, $X(\mathcal{E}_v^{\mathbf{u}}, t) = \cap_{S\in \mathcal{E}^{\mathbf{u}}_v}H_S\cap X_0$ is an intersection of finitely many convex, closed sets, hence, it is convex and closed.
\end{proof}

\begin{lemma}\label{lemma:monotone}
Given a utility function $\mathbf{u}$ and a game $v\in \mathcal{G}^{N, \mathcal{A}}$, take $x_1, x_2 \in \mathbb{R}^N$ and $S\in \mathcal{A}$. If $u_S\circ e_v(S, x_1) < u_S\circ e_v(S, x_2)$,  then for every $\lambda \in (0, 1)$ it holds that 
\begin{equation*}
u_S\circ e_v(S, x_1) < u_S\circ e_v(S, \lambda x_1 + (1-\lambda) x_2) < u_S\circ e_v(S, x_2).
\end{equation*}
\end{lemma}

\begin{proof}
Let $\lambda \in (0, 1)$, then 

\begin{equation*}
(\lambda x_1 + (1- \lambda) x_2)(S) = \lambda x_1(S) + (1- \lambda)x_2(S).
\end{equation*}
Therefore, 
\begin{equation*}
\begin{split}
e_v(S, \lambda x_1 + (1- \lambda) x_2) = v(S)-(\lambda x_1 + (1- \lambda) x_2)(S) \\
= \lambda v(S) + (1- \lambda) v(S) - (\lambda x_1(S) + (1- \lambda)x_2(S)) \\
= \lambda e_v(S, x_1) + (1- \lambda)e_v(S, x_2).
\end{split}
\end{equation*}

Since $u_S$ is strictly monotone increasing we have that
\begin{equation*}
u_S\circ e_v(S, x_1) < u_S\circ e_v(S, \lambda x_1 + (1-\lambda) x_2) < u_S\circ e_v(S, x_2).
\end{equation*}
\end{proof}

\begin{lemma}\label{lemma:X1=X1'}
Given a utility function $\mathbf{u}$ and a  game $v\in \mathcal{G}^{N, \mathcal{A}}$, if $v$ is $\mathbf{u}$-balanced, then $X(\mathcal{A}^\ast, t_1) = X(\mathcal{E}_v^{\mathbf{u}}, t_1)$.
\end{lemma}

\begin{proof}
Since $\mathcal{E}_v^{\mathbf{u}} \subseteq \mathcal{A}^\ast$, it holds that $X(\mathcal{A}^\ast, t_1) \subseteq X(\mathcal{E}_v^{\mathbf{u}}, t_1)$, .

Indirectly assume that $\exists x^{1}\in X(\mathcal{E}_v^{\mathbf{u}}, t_1) \setminus X(\mathcal{A}^\ast, t_1)$. This means that $\exists S \in \mathcal{A}^*$ such that

\begin{equation}\label{eq:e>t}
u_S \circ e_v(S, x^1) > t_1.
\end{equation}

Let $\mathcal{S}_{x^1} = \{ S\in \mathcal{A}^* \colon u_S\circ e_v(S, x^1) > t_1\}$. Then, for all $S\in \mathcal{S}_{x^1}$ it holds that $S\notin \mathcal{E}_v^{\mathbf{u}}$.

Let $x^\ast \in X(\mathcal{A}^\ast, t_1)$ be the closest point of set $X(\mathcal{A}^\ast, t_1)$ to point $x^1$. It is clear that such $x^\ast$ exists, since $X(\mathcal{A}^\ast, t_1)$ is nonempty and closed according to Lemma \ref{lemma:convex}. 


Since $X(\mathcal{E}_v^{\mathbf{u}}, t_1)$ is a convex set, for every $\lambda \in [0,1]$ it holds that $\lambda x^\ast + (1-\lambda) x^1 \in X(\mathcal{E}_v^{\mathbf{u}}, t_1)$.

By Lemma \ref{lemma:rest1} we have that for each $S\in \mathcal{S}_{x^1}$ $\exists \mathcal{B}^\ast_S\subseteq \mathcal{E}_v^{\mathbf{u}}$, $\mathcal{B}^\ast_S\in \mathcal{D}_{S^1}^{\mathcal{A}^\ast}$ such that 

\begin{equation}\label{eq:choice_of_B*}
u_{S}\circ e_v(S, x^\ast) \leq \sum_{T\in \mathcal{B}^*_S}u_T\circ e_v(T, x^\ast).
\end{equation}

Since $t_1 \leq 0$ and for all $x\in X(\mathcal{E}_v^{\mathbf{u}}, t_1)$ and $T\in \mathcal{E}_v^{\mathbf{u}}$ it holds that $u_T\circ e_v(T, x) \leq t_1$, therefore for all $S\in \mathcal{S}_{x^1}$

\begin{equation}\label{eq:e<=t}
u_{S}\circ e_v(S, x^\ast) \leq \sum_{T\in \mathcal{B}^\ast_S}u_T\circ e_v(T, x^\ast) \leq t_1 \ .
\end{equation}

By \eqref{eq:e>t}, \eqref{eq:e<=t} and the continuity of $u_S$, for every $S\in \mathcal{S}_{x^1}$ there exists $\lambda_S \in [0,1]$ such that

\begin{equation*}
u_S\circ e_v(S, \lambda_S x^\ast + (1-\lambda_S) x^1) = t_1 \ .
\end{equation*}

Let $S^1 \in \argmin_{S\in \mathcal{S}_{x^1}} || x^\ast - (\lambda_S x^\ast + (1-\lambda_S) x^1)) ||$, and let $x^2 = \lambda_{S^1} x^\ast + (1-\lambda_{S^1}) x^1$.

Then, $$u_{S^1}\circ e_v(S^1, x^2) \geq \sum_{T\in \mathcal{B}^*_{S^1}} u_T\circ e_v(T, x^2),$$ because $t_1 \leq 0$ and $u_T\circ e_v(T, x^2) \leq t_1$ for all $T\in \mathcal{B}^*_{S^1}$.

Then, there are two cases:

Case 1:
\begin{equation*}
u_{S^1}\circ e_v(S^1, x^2) = \sum_{T\in \mathcal{B}^\ast_{S^1}} u_T\circ e_v(T, x^2).
\end{equation*}
In this case, $t_1 = u_{S^1}\circ e_v(S^1, x^2) = \sum_{T\in \mathcal{B}^*_{S^1}} u_T\circ e_v(T, x^2) \leq |\mathcal{B}^*_{S^1}|t_1 \leq t_1$; therefore, $t_1 = 0.$

Then, $\sum_{T\in \mathcal{B}^*_{S^1}} u_T\circ e_v(T, x^2) = 0$, hence, for all $T\in \mathcal{B}^*_{S^1}$ it holds that $u_T\circ e_v(T, x^2) = 0$.

We know that $u_{S^1}\circ e_v(S^1, x^1) > t_1  \geq \sum_{T\in \mathcal{B}^*_{S^1}} u_T\circ e_v(T, x^1)$ and that $x^2(S^1) = \sum_{T\in \mathcal{B}^*_{S^1}}x^2(T)$ and  $x^1(S^1) = \sum_{T\in \mathcal{B}^*_{S^1}}x^1(T)$. Then, $x^1(S^1) < x^2(S^1)$, but then, $\exists T'\in \mathcal{B}^\ast_{S^1}$ such that $x^1(T') < x^2(T')$. However, $u_{T'}$ is a strictly monotone increasing function; hence, we have that $u_{T'}\circ e_v(T', x^1)>t_1$, which is a contradiction, because $x^1 \in X(\mathcal{E}_v^{\mathbf{u}}, t_1)$.

Case 2:
\begin{equation}\label{eq:2_>}
u_{S^1}\circ e_v(S^1, x^2) > \sum_{T\in \mathcal{B}^*_{S^1}} u_T\circ e_v(T, x^2) \ .
\end{equation}

By the choice of $x^2$ and Lemma \ref{lemma:monotone}, for all $S\in \mathcal{S}_{x^1}$
\begin{equation*}
u_S\circ e_v(S, x^2) \leq t_1 \ .
\end{equation*}
Since for all $S\in \mathcal{A}^*\setminus \mathcal{S}_{x_1}$ it holds that $u_S\circ e_v(S, x^1) \leq t_1$, by Lemma \ref{lemma:monotone}
\begin{equation*}
u_S\circ e_v(S, x^2) \leq t_1 \ .
\end{equation*}

This means that for all $S\in \mathcal{A}^*$ it holds that $u_S\circ e_v(S, x^2) \leq t_1$; hence, $x^2 \in X(\mathcal{A}^*, t_1)$.

Since $x^\ast \in X(\mathcal{A}^\ast, t_1)$ is the closest point of set $X(\mathcal{A}^\ast, t_1)$ to point $x^1$, we have that $x^2 = x^\ast$. However, then, \eqref{eq:2_>} contradicts \eqref{eq:choice_of_B*}.
\end{proof}

\begin{lemma}\label{lemma:t<t_1:X1=X1'}
Given a utility function $\mathbf{u}$, a $\mathbf{u}$-balanced game $v\in \mathcal{G}^{N, \mathcal{A}}$ and $t\in \mathbb{R}$, if $t<t_1$, then $X(\mathcal{A}^\ast, t) = X(\mathcal{E}_v^{\mathbf{u}}, t) = \emptyset$.
\end{lemma}

\begin{proof}
By the definition of $t_1$, we have that $X(\mathcal{A}^\ast, t) = \emptyset$ and we know that $X(\mathcal{A}^\ast, t) \subseteq X(\mathcal{E}_v^{\mathbf{u}}, t)$.

By Lemma \ref{lemma:X(t')<=X(t")}, since $t<t_1$, we have that $X(\mathcal{E}_v^{\mathbf{u}}, t) \subseteq X(\mathcal{E}_v^{\mathbf{u}}, t_1)$. Furthermore, by Lemma \ref{lemma:X1=X1'} it holds that $X(\mathcal{E}_v^{\mathbf{u}}, t_1) = X(\mathcal{A}^\ast, t_1)$. Since $t_1\leq 0$, it holds that $X(\mathcal{A}^\ast, t_1) \subseteq X(\mathcal{A}^\ast, 0) = \mathbf{u}\text{-core}(v)$. Therefore,
\begin{equation*}
X(\mathcal{E}_v^{\mathbf{u}}, t) \subseteq X(\mathcal{E}_v^{\mathbf{u}}, t_1) = X(\mathcal{A}^*, t_1) \subseteq X(\mathcal{A}^*, 0) = \mathbf{u}\text{-core}(v) \ .
\end{equation*}

Then, for every $x\in X(\mathcal{E}_v^{\mathbf{u}}, t)$ and for every $S\in \mathcal{A}^\ast\setminus \mathcal{E}_v^{\mathbf{u}}$ by Lemma \ref{lemma:rest1} and by the non-positivity of $t$ it holds that $\exists \mathcal{B}\in \mathcal{D}_S^{\mathcal{A}^\ast}$, $\mathcal{B} \subseteq \mathcal{E}_v^{\mathbf{u}}$ such that

\begin{equation*}
u_S\circ e_v(S, x) \leq \sum_{T\in \mathcal{B}} u_T \circ e_v(T, x) \leq t \ .
\end{equation*}

Then, $x\in X(\mathcal{A}^\ast, t)$, that is, $X(\mathcal{E}_v^{\mathbf{u}}, t) \subseteq X(\mathcal{A}^\ast, t)$. Summing up, we can can conclude that $X(\mathcal{E}_v^{\mathbf{u}}, t) = X(\mathcal{A}^\ast, t) = \emptyset$.
\end{proof}

\begin{remark}\label{rem:szukmag}
Notice, that all the results we have discussed so far hold if one defines $\mathbf{u}$-essentiality by the elements of the $\mathbf{u}$-least core instead of the elements of the $\mathbf{u}$-core.
\end{remark}

The following proposition is a consequence of Lemmata \ref{lemma:X1=X1'} and \ref{lemma:t<t_1:X1=X1'}.

\begin{proposition}\label{prop:X1=Y1}
The following holds: $t_1 = t_1'$ and $X_1 = X_1'$.
\end{proposition}

\begin{proof}
By definition $t_1' = \min \{ t\colon X(\mathcal{E}_v^{\mathbf{u}}, t) \neq \emptyset \}$. We know, that $t_1'\leq t_1$; therefore, by Lemmata \ref{lemma:X1=X1'} and \ref{lemma:t<t_1:X1=X1'} we have that $X_1' = X(\mathcal{E}_v^{\mathbf{u}}, t_1') =  X(\mathcal{A}^\ast, t_1')$. However,  $X(\mathcal{A}^\ast, t_1') \neq \emptyset$ if and only if $t_1'\geq t_1$, hence $t_1 = t_1'$ and $X_1 = X_1'$.
\end{proof}

\begin{lemma}\label{lemma:rest2}
Let $k$ be a positive integer, $S \in \mathcal{A}^\ast  \setminus \cup_{r=1}^{k-1} W_r$ be such that $\mathcal{D}_S^{\mathcal{A}^{\ast}} \neq \emptyset$, and $\mathcal{B}^\ast \in \mathcal{D}_S^{\mathcal{A}^{\ast}}$. Then, there exists a coalition $T^\ast \in \mathcal{B}^\ast$ such that $T^\ast \notin \cup_{r=1}^{k-1} W_r$.
\end{lemma}

\begin{proof}
If $k = 1$, then $ \cup_{r=1}^{k-1} W_r = \emptyset$, hence, $T^\ast \notin \cup_{r=1}^{k-1} W_r$. 

\noindent If $k \geq 2$, then indirectly assume that $\mathcal{B}^\ast \subseteq \cup_{r=1}^{k-1}W_r$. Then for every $x \in X_{k-1}$

\begin{equation*}
\begin{split}
u_S \circ e(S, x) = u_S(v(S) - x(S)) = u_S \left( v(S) -  \sum_{T \in \mathcal{B}^\ast} x (T) \right) \\
= u_S \left( v(S) - \sum \limits_{T \in \mathcal{B}^\ast} (v(T) - u_T^{-1}(c_T)) \right) \ .
\end{split}
\end{equation*}

\noindent Therefore, for each $x,x' \in X_{k-1}$ it holds that $u_S \circ e(S,x) = u_S \circ e(S,x')$, meaning that $S \in \cup_{r=0}^{k-1}W_r$, which is a contradiction.
\end{proof}

The following theorem generalizes   \citet{Huberman1980}'s theorem (Theorem 7 on page 420 of \cite{Huberman1980}):

\begin{theorem}\label{th:u-Huberman}
Consider a $\mathbf{u}$-balanced game $v \in \mathcal{G}^{N,\mathcal{A}}$, and let

\begin{equation*}
Y_1 = \{ x \in I^\ast (v) \colon u_S \circ e (S,x) \leq t_1, \  \forall S \in \mathcal{E}_v^{\mathbf{u}} \} ,
\end{equation*}

\noindent and for all $k \geq 2$ let $Y_k$ be defined as follows: 

\begin{equation*}
Y_k = \{ x \in X_{k-1} \colon u_S \circ e(S,x) \leq t_k, \ \forall S \in \mathcal{E}_v^{\mathbf{u}} \setminus (\cup_{r=1}^{k-1} W_r) \} .
\end{equation*}

\noindent Then, $X_k = Y_k$ for all $k\geq 1$.
\end{theorem}

In other words, Theorem \ref{th:u-Huberman} claims that the $\mathbf{u}$-essential coalitions give a characterization set for the $\mathbf{u}$-prenucleolus of $\mathbf{u}$-balanced games.

\begin{proof}
First, notice that $X_k \subseteq Y_k$ holds for all $k$ by definition.

By Proposition \ref{prop:X1=Y1}, we have that $X_1 = Y_1$; therefore, $X_1, Y_1 \subseteq \mathbf{u}\text{-core}(v)$.

Suppose for contradiction that there exists $k > 1$ such that $X_k \nsupseteq Y_k$, which means that there exist $y^\ast \in Y_k$ and $S \in \mathcal{A}^\ast \setminus (\mathcal{E}_v^{\mathbf{u}} \cup (\cup_{r=1}^{k-1} W_r))$ such that $u_S \circ e (S,y^\ast) > t_k$.  


By Lemma \ref{lemma:rest1}  for each $x\in X_1$ there exists $\mathcal{B}_x \in \mathcal{D}_S^{\mathcal{A}^{\ast}} \cap \mathcal{E}_v^{\mathbf{u}}$ such that $u_S \circ e(S, x) \leq \sum_{T \in \mathcal{B}_x} u_T \circ e (T, x)$.


Then

\begin{equation*}
\begin{split}
u_S \circ e (S, y^\ast)  &\leq \sum_{T\in \mathcal{B}_{y^*}} u_T \circ e(T, y^\ast) \ .
\end{split}
\end{equation*}

Moreover, by definition $u_S \circ e(S, x)\leq 0$ for all $S\in \mathcal{A}^{\ast}, x\in X_{k-1}$.  Therefore, for any coalition $T\in \mathcal{B}_{y^*}$ we have that $u_S \circ e(S, y^*)\leq u_T \circ e(T, y^*)$. By Lemma \ref{lemma:rest2}, there exists $T^*\in \mathcal{B}_{y^*}$ such that $T^*\notin \cup _{r=1}^{k-1}W_r$. Therefore, $u_S \circ e(S, y^*)\leq u_{T^*} \circ e(T^*, y^*)\leq t_k$, which is a contradiction.
\end{proof}

In words, Theorem \ref{th:u-Huberman} gives a characterization set for the $\mathbf{u}$-prenucleolus in case of $\mathbf{u}$-balanced games. As a direct corollary of this theorem, we can find a characterization set for the percapita prenucleolus in case of balanced games (notice, that the core and the percapita core coincide), or we can shift the values of the non-trivial coalitions uniformly so that the $\mathbf{u}$-core ($\mathbf{u}$ here is the shift) of the game becomes nonempty, while the $\mathbf{u}$-prenucleolus coincides with the prenucleolus. We discuss these applications in more detail in Section \ref{sec:invariance}.

\section{Two invariance results}\label{sec:invariance}

There are $\mathbf{u}$ functions such that  the $\mathbf{u}$-prenucleolus of a game is the same as its prenucleolus, while the $\mathbf{u}$-core of the game is different from its core. These utility functions can be useful in finding a characterization set for the prenucleolus of non-balanced games. 

In this section, we characterize the classes of $\mathbf{u}$-functions under which the $\mathbf{u}$-prenucleolus and the $\mathbf{u}$-core are the same as the prenucleolus and the core, respectively.

\begin{lemma}\label{lemma:invariance:prenuc}
Given a game $v \in \mathcal{G}^{N,\mathcal{A}}$, and utility functions $\mathbf{u}^1$ and $\mathbf{u}^2$, the $\mathbf{u}^1$-prenucleolus coincides with the $\mathbf{u}^2$-prenucleolus if for every $S, T \in \mathcal{A}^*$ and $x\in I^*(v)$

\begin{equation} \label{eq:lemma_normal}
u^1_S \circ e(S, x) \leq u^1_T \circ e(T, x) 
\end{equation}

\noindent if and only if

\begin{equation} \label{eq:lemma_u}
u^2_S \circ e(S, x) \leq u^2_T \circ e(T, x) .
\end{equation}
\end{lemma}

\begin{proof}
Notice, that w.l.o.g. we can assume that $\mathbf{u}^1 = \mathbf{id}$. Then let $\mathbf{u}$ denote $\mathbf{u}^2$.

Let $x, y \in I^*(v)$ be such that $E(x) \leq _L E(y)$, where
\begin{equation*}
E(x) := [e(S_1, x), e(S_2, x), \dots , e(S_{|\mathcal{A}^*|}, x)]
\end{equation*}
\begin{equation*}
E(y) := [e(T_1, y), e(T_2, y), \dots , e(T_{|\mathcal{A}^*|}, y)] .
\end{equation*}

\bigskip

Case 1: $E(x) = E(y)$: By \eqref{eq:lemma_normal} and \eqref{eq:lemma_u} $e(S_n, x) = e(T_n, y)$ for all $1\leq n \leq |\mathcal{A}^*|$ is equivalent with $u_{S_n}\circ e(S_n, x) = u_{T_n} \circ e(T_n, y)$  for all $1\leq n \leq |\mathcal{A}^*|$. Meaning $E_{\mathbf{u}}(x) = E_{\mathbf{u}}(y)$.

Case 2: $E(x) \neq E(y)$: Then, there exists $k$ such that

\begin{equation}\label{kell}
\begin{split}
e(S_n, x) = e(T_n, y) \ \forall n<k, \\
e(S_k, x) < e(T_k, y).
\end{split}
\end{equation}

By \eqref{eq:lemma_normal} and \eqref{eq:lemma_u} we have that \eqref{kell} is equivalent with

\begin{equation}
\begin{split}
u_{S_n} \circ e(S_n, x) = u_{T_n} \circ e(T_n, y) \ \forall n<k, \\
u_{S_k} \circ e(S_k, x) < u_{T_k} \circ e(T_k, y).
\end{split}
\end{equation}

This proves that for each $x, y \in I^*(v)$ $E(x) \leq _L E(y)$ if and only if $E_{\mathbf{u}}(x) \leq _L E_{\mathbf{u}}(y)$. Therefore, $x\in N^*(v)$ if and only if $x\in N_{\mathbf{u}}^*(v)$.
\end{proof}

For example, if $\mathbf{u}$ is such that $u_S = u_T$ for all $S, T \in \mathcal{A}^\ast$, then the $\mathbf{u}$-prenucleolus coincides with the prenucleolus.

\begin{example}
In this example, we show how to find a characterization set for the prenucleolus of non-balanced games using Lemma \ref{lemma:invariance:prenuc}.

Let $v\in \mathcal{G}^{N, \mathcal{A}}$ be a game and let $\mathbf{u}$ be the utility function, such that $u_S \circ e(S, x) := e(S, x) - \varepsilon^*$ for all $S\in \mathcal{A}^*$, where $\varepsilon ^*$ is the optimum of the following LP:

\begin{equation}
\begin{array}{llr}
& \varepsilon \to \min & \\
\textup{s.t.} &  e(S, x) \leq \varepsilon, & \ S \in \mathcal{A}^\ast \\
& x \in I^\ast (v) &  \\
& \varepsilon \in \mathbb{R}. & 
\end{array}
\end{equation}

In other words, $\varepsilon ^*$ is such that the least core is the $\varepsilon^*$-core.

Then, by Lemma \ref{lemma:invariance:prenuc} the $\mathbf{u}$-prenucleolus of the game coincides with the prenucleolus. In addition, the game is $\mathbf{u}$-balanced because the least core is never empty.

Therefore, by Theorem \ref{th:u-Huberman}, in the case of this $\mathbf{u}$ function ($u_S (t) = t - \varepsilon^\ast$, $S \in \mathcal{A}^\ast$), the $\mathbf{u}$-essential coalitions form a characterization set for the prenucleolus of $v$ even if $v$ is not balanced.
\end{example}

\bigskip

In the following we consider the equivalence of the core and the $\mathbf{u}$-core.

\begin{lemma}\label{lemma:invariance:core}
Given a game $v \in \mathcal{G}^{N,\mathcal{A}}$, and utility functions $\mathbf{u}^1$ and $\mathbf{u}^2$, the $\mathbf{u}^1$-core coincides with the $\mathbf{u}^2$-core if for every $S\in \mathcal{A}^*$ and $x\in I^*(v)$

\begin{equation}\label{eq:incariance:core1}
u_S^1 \circ e(S, x) \leq 0
\end{equation}

\noindent if and only if

\begin{equation}\label{eq:incariance:core2}
u_S^2 \circ e(S, x) \leq 0 .
\end{equation}
\end{lemma}

\begin{proof}
$\mathbf{u}^1$-core$(v) = \{ x\in I^*(v) \colon u_S^1\circ e(S, x) \leq 0 \ \forall S\in \mathcal{A}^*\}$. Due to the equivalence of \eqref{eq:incariance:core1} and \eqref{eq:incariance:core2} this equals $\{ x\in I^*(v) \colon u_S^2\circ e(S, x) \leq 0 \ \forall S\in \mathcal{A}^*\} = \mathbf{u}^2$-core$(v)$.
\end{proof}

For example, if $\mathbf{u}$ is such that $u_S(0) = 0$ for all $S\in \mathcal{A}^*$, then the $\mathbf{u}$-core coincides with the core; therefore a game is balanced if and only if it is $\mathbf{u}$-balanced.

\begin{example}
In this example, we show how to find a characterization set for the percapita prenucleolus of balanced games.

Let $v\in \mathcal{G}^{N,\mathcal{A}}$ be a game and let $\mathbf{u}$ be the percapita function, that is, $u_S \circ e(S, x) = \frac{e(S, x)}{|S|}$ for all $S\in \mathcal{A}^\ast$.

By Lemma \ref{lemma:invariance:core}, the $\mathbf{u}$-core coincides with the core.

Then, a coalition $S\in \mathcal{A}^\ast$ is $\mathbf{u}$-essential (see Definition \ref{def:u-ess}), if either $\mathcal{D}_S^{\mathcal{A}^\ast} = \emptyset$, (if $\mathcal{A} = \mathcal{P}(N)$, these coalitions are the singletons) or if there exists $x\in \core (v)$ such that

\begin{equation*}
\frac{e(S, x)}{|S|} > \max_{\mathcal{B}\in \mathcal{D}_S^{\mathcal{A}^*}}\sum_{T\in \mathcal{B}}\frac{e(T, x)}{|T|} \,.
\end{equation*} 

By Theorem \ref{th:u-Huberman}, the $\mathbf{u}$-essential (percapita-essential) coalitions form a characterization set for the percapita prenucleolus in case of balanced games.
\end{example}

\section{An example}\label{sec:application}

There are certain classes of games and utility functions for which there are only polynomial many $\mathbf{u}$-essential coalitions in the number of players. For example, consider the class of assignment games with the reciprocal percapita utility function $\mathbf{u}$. Our definition of the the reciprocal percapita utility function $\mathbf{u}$ is that for all $v\in \mathcal{G}^{N}, S\in \mathcal{P}^*(N)$ we have that $u_S\circ e(S, x) = |S|e(S, x)$. The reciprocal percapita utility function can be "interpreted" in a way that the value of a coalition is a public value for the members of the coalition. Therefore, each player's utility is the excess of the coalition; hence, the total utility of the excess of the coalition is the excess of the coalition multiplied by the size of the coalition.

In case of an assignment game, there are sellers ($M'$) and buyers ($M$). Each seller $j\in M'$ has a reservation value of $c_j\geq 0$ and each buyer $i\in M$ values the object of seller $j$ to $h_{i, j}\geq 0$. If a buyer and a seller trade, they make a joint profit of $a_{i, j} = \max \{ 0, h_{i, j} - c_j \}$. These joint profits can be displayed in an assignment matrix $A$:

\begin{equation*}
A = 
\left[
\begin{array}{@{}rrrr@{}}
a_{1, 1} & a_{1, 2} & \dots & a_{1, m'} \\
a_{2, 1} & a_{2, 2} & \dots & a_{2, m} \\
\vdots&\vdots&\ddots&\vdots\\
a_{m, 1} & a_{m, 2} & \dots & a_{m, m'}
\end{array}
\right]
\end{equation*}

A matching $\mu '$ is a subset of $M \times M'$, where each agent appears in at most one pair. Let $\mathcal{M}(M, M')$ be the set of matchings.

An assignment game has the set of players $M \cup M'$ and a characteristic function $w_A$ defined as: for all $S \subseteq M, T \subseteq M'$

\begin{equation*}
w_A (S\cup T) = \max \{ \sum_{(i, j)\in \mu} a_{i, j} \colon \mu \in \mathcal{M}(S, T)\}.
\end{equation*}

The core can be described using only the matchings and the singletons in the following way: 

\begin{equation*}
\begin{split}
\text{core}(w_A) = \Big\{ (x, y) \in \mathbb{R}^M \times \mathbb{R}^{M'} \colon \sum_{i \in M} x_i + \sum_{j \in M'} y_j = w_A(M \cup M'), \\ 
x_i + y_j \geq a_{i, j} \ \forall (i, j )\in M\times M', x_i \geq 0 \ \forall i \in M, y_j \geq 0 \ \forall j \in M' \Big\}.
\end{split}
\end{equation*}

Moreover, the core of an assignment game is non-empty \citep{ShapleyShubik1972}. By Lemma \ref{lemma:invariance:core}, the core of an assignment game coincides with the $\mathbf{u}$-core of the game in case of the reciprocal percapita utility function.

Find the $\mathbf{u}$-essential coalitions in case of assignment games. The singletons are $\mathbf{u}$-essential by definition. The matchings are $\mathbf{u}$-essential, but the other pairs are not $\mathbf{u}$-essential. Indeed, let $i, j \in N$, $\{ i, j \} \notin \mathcal{M}(M, M')$, then for each $x\in \mathbf{u}$-core$(w_A)$ we have that  $2(w_A(\{ i, j \} ) - x_i - x_j) \leq -x_i-x_j$, because $0\leq x_i + x_j$.

Consider a coalition $S\in \mathcal{P}^*(N)$ with cardinality larger than two. Let $x\in \mathbf{u}$-core$(w_A)$ and $\mu^*$ be an optimal matching for $S$. The left-hand side of the inequality in Definition \ref{def:u-ess} is:

\begin{equation*}\label{eq:left}
\begin{split}
u_S\circ e(S, x) = |S|(w_A(S)-x(S))
= |S|\left(\sum_{(i, j)\in \mu^*} (a_{i, j} - x_i - x_j) - \sum_{\substack{ k \in N,\\ (k, \cdot )\notin \mu^*,\\ (\cdot , k) \notin \mu^*}} x_k \right) \, .
\end{split}
\end{equation*}

\noindent Moreover, for $\mathcal{B}^* =  \mu^* \cup \{ \{ k \} \}_{ k \in N, (k, \cdot )\notin \mu^*, (\cdot , k) \notin \mu^*}$ the right-hand side of the inequality in Definition \ref{def:u-ess} is

\begin{equation}\label{eq:right}
\sum_{T\in \mathcal{B}^*} u_T\circ u(T, x) = 2 \left( \sum_{(i, j)\in \mu^*} (a_{i, j} - x_i - x_j) \right) - \sum_{\substack{ k \in N,\\ (k, \cdot )\notin \mu^*,\\ (\cdot , k) \notin \mu^*}} x_k \, .
\end{equation}

\noindent Subtract $2(\sum_{(i, j)\in \mu^*} (a_{i, j} - x_i - x_j)) - |S|\sum_{\substack{ k \in N,\\ (k, \cdot )\notin \mu^*,\\ (\cdot , k) \notin \mu^*}} x_k$ from both sides. Then we get

\begin{equation*}
(|S|-2) \sum_{(i, j) \in \mu^*} a_{i, j}
\end{equation*}

\noindent on the left-hand side, and

\begin{equation*}
(|S|-2)\sum_{(i, j) \in \mu^*}( x_i + x_j ) + (|S| - 1) \sum_{\substack{ k \in N,\\ (k, \cdot )\notin \mu^*,\\ (\cdot , k) \notin \mu^*}} x_k
\end{equation*}

\noindent on the right-hand side. Since $x\in \mathbf{u}$-core$(w_A)$, it holds that $\sum_{(i, j) \in \mu^*}( x_i + x_j ) \geq \sum_{(i, j) \in \mu^*} a_{i, j}$ and $\sum_{\substack{ k \in N,\\ (k, \cdot )\notin \mu^*,\\ (\cdot , k) \notin \mu^*}} x_k \geq 0$. Then, it follows that 

\begin{equation*}
u_S\circ e(S, x) \leq \sum_{T\in \mathcal{B}^*} u_T\circ u(T, x) \, .
\end{equation*}

\noindent It means that there is no $x\in \mathbf{u}$-core$(w_A)$ such that the left-hand side would be strictly larger than the right-hand side, hence $S$ is not $\mathbf{u}$-essential.

In conclusion, only the singletons and the matchings are $\mathbf{u}$-essential in case of assignment games, hence, there are only polynomial many $\mathbf{u}$-essential coalitions in the number of players.

\section{Conclusion}

We have introduced a generalization of the prenucleolus using utility functions, namely the $\mathbf{u}$-prenucleolus. This generalization also generalizes the percapita prenucleolus \citep{Grotte1970, Grotte1972} and the $q$-nucleolus \citep{Solymosi2019}. On the other hand, the $\mathbf{u}$-prenucleolus is a special case of the general prenucleolus \citep{PottersTijs1992, MaschlerPottersTijs1992}.

We have considered TU-games with restricted cooperation. For such games, some of the original properties of the prenucleolus change: for example, the prenucleolus is no longer a single-valued solution. \citet{Katsev2013} gave necessary and sufficient conditions for the prenucleolus to be non-empty and to be single-valued, respectively. We have generalized these results to the $\mathbf{u}$-prenucleolus.

Using the idea of utility functions, we have also introduced generalizations of the core, least core, balanced games and essential coalitions: the $\mathbf{u}$-core, $\mathbf{u}$-least-core, $\mathbf{u}$-balanced games and $\mathbf{u}$-essential coalitions, respectively. We have generalized the  Bondareva--Shapley theorem \citep{Bondareva1963,Shapley1967,Faigle1989} by showing that a game is $\mathbf{u}$-balanced if and only if its $\mathbf{u}$-core is not empty. We have also generalized Huberman's theorem \citep{Huberman1980}; we have shown that $\mathbf{u}$-essential coalitions form a characterization set of the $\mathbf{u}$-prenucleolus in case of $\mathbf{u}$-balanced games.

We have given sufficient conditions on the utility functions for the $\mathbf{u}$-prenucleolus and the $\mathbf{u}$-core to be invariant. Finally, we have discussed a class of games and a utility function, where a game has polynomial many $\mathbf{u}$-essential coalitions.

Future work can include other similar generalizations of the prenucleolus, for example a one where the utility function is not applied to the excesses, but to the coalitional values and/or to the preimputations. Another possible direction of future research is the deeper understanding of the $\mathbf{u}$-prenucleolus by considering dual games.


\end{document}